\documentclass[12pt, reqno]{amsart}
\usepackage{amsmath, amsthm, amscd, amsfonts, amssymb, graphicx, color, mathrsfs}
\usepackage[bookmarksnumbered, colorlinks, plainpages]{hyperref}
\usepackage[all]{xy} 
\usepackage{slashed}
\usepackage{cite}

\newtheorem{theorem}{Theorem}[section]

\newtheorem{proposition}[theorem]{Proposition}
\newtheorem{corollary}[theorem]{Corollary}

\theoremstyle{definition}
\newtheorem{definition}[theorem]{Definition}
\newtheorem{example}[theorem]{Example}

\newtheorem{openproblem}[theorem]{Open problem}

\theoremstyle{remark}
\newtheorem{remark}[theorem]{Remark}
\numberwithin{equation}{section}

\begin{document}
\setcounter{page}{1}

\title[Schatten properties for H\"ormander classes on compact Lie groups]{Schatten-von Neumann properties for H\"ormander classes on compact Lie groups}

\author[D. Cardona]{Duv\'an Cardona}
\address{
  Duv\'an Cardona:
  \endgraf
  Department of Mathematics: Analysis, Logic and Discrete Mathematics
  \endgraf
  Ghent University, Belgium
  \endgraf
  {\it E-mail address} {\rm duvan.cardonasanchez@ugent.be}
  }
  
  \author[M. Chatzakou]{Marianna Chatzakou}
\address{
  Marianna Chatzakou:
  \endgraf
  Department of Mathematics: Analysis, Logic and Discrete Mathematics
  \endgraf
  Ghent University, Belgium
  \endgraf
  {\it E-mail address} {\rm Marianna.Chatzakou@UGent.be}
  }

\author[M. Ruzhansky]{Michael Ruzhansky}
\address{
  Michael Ruzhansky:
  \endgraf
  Department of Mathematics: Analysis, Logic and Discrete Mathematics
  \endgraf
  Ghent University, Belgium
  \endgraf
 and
  \endgraf
  School of Mathematical Sciences
  \endgraf
  Queen Mary University of London
  \endgraf
  United Kingdom
  \endgraf
  {\it E-mail address} {\rm michael.ruzhansky@ugent.be}
  }

\author[J. Toft]{Joachim Toft}
\address{
  Joachim Toft:
  \endgraf
   Department of Mathematics
  \endgraf
  Linnæus University
  \endgraf
  V\"axj\"o-Sweden
  \endgraf
    {\it E-mail address} {\rm joachim.toft@lnu.se }}

\thanks{The authors were supported  by the FWO  O\textnormal{d}ysseus  1  grant  G.0H94.18N:  Analysis  and  Partial Differential Equations, by the Methusalem programme of the Ghent University Special Research Fund (BOF)
(Grant number 01M01021) and by the  \textnormal{d}yCon Project 2015 H2020-694126. Marianna Chatzakou is also supported by the FWO Fellowship grant No 12B1223N.  Michael Ruzhansky is also supported  by EPSRC grant 
EP/R003025/2.
}


\begin{abstract}
Let $G$ be a compact Lie group of dimension $n.$ In this work we characterise the membership of classical pseudo-differential operators on $G$ in the trace class ideal $S_{1}(L^2(G)),$ as well as in the setting of the Schatten ideals $S_{r}(L^2(G)),$  for all $r>0.$ In particular, we deduce Schatten characterisations of
elliptic pseudo-differential operators of $(\rho,\delta)$-type
for the large range $0\leq \delta<\rho\leq 1.$ Additional necessary and sufficient conditions are given in terms of the matrix-valued symbols of the operators, which are global functions on the phase space $G\times \widehat{G},$ with the momentum variables belonging to the unitary dual $\widehat{G}$ of $G$.
%
%
%
%
In terms of the parameters $(\rho,\delta),$ on the torus $\mathbb{T}^n,$ we demonstrate the sharpness of our results showing the existence of atypical operators in the exotic class $\Psi^{-\varkappa}_{0,0}(\mathbb{T}^n),$ $\varkappa>0,$ belonging to all the Schatten ideals. Additional order criteria are given in the setting of classical pseudo-differential operators. We present also some open problems in this setting.
\end{abstract} 

\maketitle

\tableofcontents
\allowdisplaybreaks

\section{Introduction}

\subsection{Outline}

Let $M$ be an orientable compact manifold without boundary with volume element $\textnormal{d}x,$ and let us consider the Hilbert space $L^2(M)=L^2(M,\textnormal{d}x)$. For any $m\in \mathbb{R}$ and for $0\leq \delta<\rho\leq 1,$ let $\Psi^m_{\rho,\delta}(M)$ be the H\"ormander class of continuous linear operators on $C^\infty(M),$ in local coordinates having the form
\begin{equation}
    Af(x)=\smallint_{\mathbb{R}^n}\smallint_{\mathbb{R}^n}e^{2\pi i(x-y,\theta)}a(x,\theta)f(y)\textnormal{d}y\textnormal{d}\theta,
\end{equation}and defined by those symbols $a:=a(x,\theta)\in C^\infty(\mathbb{R}^n\times \mathbb{R}^n)$ satisfying the estimates
\begin{equation}
    |\partial_x^\beta\partial_\xi^\alpha a(x,\theta)|\lesssim_{\alpha,\beta,K}  (1+|\theta|)^{m-\rho|\alpha|+\delta|\beta|}
\end{equation} uniformly in $x$ over compact subsets $K\subseteq \mathbb{R}^n.
$ It is
%
%
%
well-known that the class $\Psi^m_{\rho,\delta}(M)$ is invariant under changes of coordinates if $\rho>1-\delta.$ On the other hand, when $M$ has a group structure compatible with its differential structure, namely, when $M=G$ is a compact Lie group, in \cite{Ruz,RuzhanskyTurunenWirth2014} one  introduced the notion of a global symbol allowing the construction of new classes of pseudo-differential operators $\Psi^m_{\rho,\delta}(G)$ also when $0\leq\rho\leq 1-\delta,$ and providing a new description of the H\"ormander classes in the case where $\rho>1-\delta.$ In this work we investigate necessary and sufficient conditions in order to guarantee the inclusion of the H\"ormander classes on compact Lie groups in the Schatten von-Neumann classes $\mathscr{S}_{r}(L^2(G)),$ namely, we will investigate sharp conditions allowing for  the  inclusion 
\begin{equation}\label{main:problem}
 \Psi^m_{\rho,\delta}(G)\subseteq \mathscr{S}_{r}(L^2(G)).   
\end{equation}
We recall that for any $r>0,$ a compact operator $T :L^2(G)\rightarrow L^2(G) $ belongs to the Schatten von-Neumann ideal $\mathscr{S}_{r}(L^2(G)),$ if the sequence of its singular values $\{s_{n}(T)\}_{n\in \mathbb{N}}$ (formed by the eigenvalues of the operator $\sqrt{T^*T}$) belongs to $\ell^r(\mathbb{N}),$ that is, if  $\sum_{n=1}^\infty s_{n}(T)^r<\infty.$

\subsection{Historical aspects}
%
%
It is well known that an elliptic pseudo-differential operator $A\in \Psi^m_{1,0}(M)$  of order $m\in \mathbb{R},$  belongs to the ideal $S_{r}(L^2(M)),$ $r>0,$ if and only if $m<-n/r.$ Once removed the ellipticity condition the problem of finding  order criteria for classifying pseudo-differential operators on the ideal $S_{r}(L^2(M))$ is still an open problem. However, in the literature, if one considers the problem of classifying pseudo-differential in the Schatten von-Neumann classes $\mathscr{S}_{r}(L^2(\mathbb{R}^n)),$ whose symbols belong to the H\"ormander classes $S^{m}_{\rho,\delta}(\mathbb{R}^n),$ the Beals-Fefferman classes $S_{\Phi,\phi}^{M_1,M_2}(\mathbb{R}^n),$ or the H\"ormander classes $S(m,g)$  the subject becomes more classical. Indeed, in 
\cite{Hormander1979}, H\"ormander observed that the  distribution of the eigenvalues (and then the Schatten properties) of  an elliptic pseudo-differential operator $A=\textnormal{Op}^{w}(a)$ is  encoded in terms of the level sets of the symbol $a.$  Indeed, he showed that the spectral formula
\begin{equation*}
    N(\lambda)\sim \smallint_{a(x,\xi)<\lambda}\textnormal{d}x\textnormal{d}\xi,
\end{equation*}
holds for any $\lambda>0.$ Here, $N(\lambda)=\#\{j:|\lambda_j|\leq \lambda\}$ denotes the spectral function of the operator $A.$ The first results of this type can be traced back to H. Weyl for second order differential operators,
and to R. Courant (see \cite[Page 297]{Hormander1979}). Then, in Theorem 3.9 of  \cite{Hormander1979}, H\"ormander proved the following sufficient condition
\begin{equation}
    m\in L^1(\mathbb{R}^{2n}),\, a\in S(m,g)\Longrightarrow
    \textnormal{Op}^{w}(a) \in \mathscr{S}_{1}(L^2(\mathbb{R}^n)),
\end{equation}
with the metric $g$ and the weight function $m$ satisfying suitable conditions.
In \cite{Hormander1979B}, H{\"o}rmander also
characterised the $L^2$ continuity of Weyl operators
with the symbols in $S(m,g)$ as
\begin{equation}\label{Lp:Condition:RnCont}
\{ \textnormal{Op}^{w}(a) \, : \, a\in S(m,g)\, \} \subseteq
\mathscr{S}_{\infty}(L^2(\mathbb{R}^n))
\ \Longleftrightarrow \ 
m \in L^\infty .
\end{equation}

\par

By adding some additional conditions on $m$ and $g$,
Buzano and Nicola in \cite{BuzanoNicola2004}, extended \eqref{Lp:Condition:RnCont} into
\begin{equation}\label{Lp:Condition:Rn}
\{ \textnormal{Op}^{w}(a) \, : \, a\in S(m,g)\, \} \subseteq
\mathscr{S}_{p}(L^2(\mathbb{R}^n))
\ \Longleftrightarrow \ 
m \in L^p ,
\end{equation}
for every $p\in [1,\infty]$. In \cite{Toft2006}, it is shown that
\eqref{Lp:Condition:RnCont} still holds true without the
additional assumptions on $m$ and $g$ in \cite{BuzanoNicola2004}.

\par

In \cite{BuzanoToft2010} the Schatten characterization
\begin{equation}\label{Lp:Condition:RnContII}
\textnormal{Op}^{w}(a) \in
\mathscr{S}_{p}(L^2(\mathbb{R}^n))
\ \Longleftrightarrow \ 
a \in L^p ,
\end{equation}
provided $a\in S(m,g)$ and $h_g^Nm\in L^p$ for some
$N\ge 0$. Here $h_g\le 1$ is the Planck's function.
For further Schatten properties of pseudo-differential operators
on $\mathbb R^n$, see e.g.
\cite{Robert1978,Daubechis1980,GrochenigHeil,
Simon1992,Sjostrand1994,Toft2007,Toft2019} and for Schatten properties on compact manifolds we refer the reader to the works \cite{Delgado2013,DelgadoRuzhansky2014,DelRuzTrace1111,DelRuzTrace111,Delgado2015,DelgadoRuzhanskyWang,DelgadoRuzhanskyWang2,DelgadoRuzhanskyTokmagambetov2017,DelRuzTrace11,DelgadoRuzhansky2018MS,DelRuzTrace1,ChatzakouDelgadoRuzhansky2021,DelgadoRuzhansky2021,Delgado2022,ChatzakouDelgadoRuzhansky2022,Chatzakou2022,CardonaKumar2021,CardonaDelCorral2020,CardonaDelCorral2020K,CardonaKumar2019,Cardona2019,CardonaDelgadoRuzhanksyLOcalWeyl,CardonaDelgadoRuzhansky2022}.

\subsection{Exotic examples and the main result}
On the other hand, necessary and sufficient conditions of the type \eqref{Lp:Condition:Rn} for non-elliptic operators on compact Lie groups are still  an open problem, as in the case of classical pseudo-differential operators (operators with polyhomogeneous symbols) as well as in the modern setting of the $(\rho,\delta)$-classes on $G,$ (see \cite{Ruz}) allowing the complete range $0\leq \delta<\rho\leq 1.$ Although when $0\leq \rho<\delta\leq 1,$ as we will show, the order condition $m<-n/r$  assuring the membership of an elliptic operator in the class $\mathscr{S}_r(L^2(G))$ is a sharp criterion, the situation changes dramatically  if one considers the borderline $\delta=\rho=0.$ Indeed, in the case of the torus $G=\mathbb{T}^n$ with arbitrary dimension $n$ we have discovered the following strongly atypical situation:
\\

\noindent {\bf $\bullet$} {\it For any $\varkappa>0,$ there exists a {non-elliptic} pseudo-differential operator $A$ in the exotic class  $\Psi^{-\varkappa}_{0,0}(\mathbb{T}^n
)\setminus \Psi^{-\varkappa-\varepsilon}_{0,0}(\mathbb{T}^n
), $ for all $\varepsilon>0,$ that belongs to all the Schatten ideals $\mathscr{S}_r(L^2(\mathbb{T}^n)),$ with $0<r<\infty.$ }
\\[2ex]
We present later this statement  in the form of Theorem \ref{The:atypical:case}.
Here we observe that the classes $\Psi_{0,0}^{m}(G)$ on a compact Lie group $G$ are of interest in PDE when computing inverses of real vector fields $X+c$, where the constant term $c$ belongs to an exceptional set $\mathscr{C}\subseteq i\mathbb{R},$ see \cite[Page 627]{RuzhanskyWirth2015} for details.
%
%

Now, we are going to discuss our main results  and we also will propose some conjectures related to the inclusion in \eqref{main:problem} which is the central  question of this manuscript. 
To continue  let us fix the notation and let us introduce the notion of a (full/global) matrix-valued symbol as developed by the third author and Turunen in \cite{Ruz}. One reason for this is that our criteria will be addressed in terms of such matrix-valued symbols. 

Let us consider the unitary dual $\widehat{G}$ of the compact Lie group $G$, which is formed by all the equivalent classes $[\xi]$ of continuous, unitary, and irreducible representations $\xi:G\mapsto \textnormal{U}(\mathbb{C}^{\ell}),$ and let $\ell=d_\xi$ be the dimension of the representation space. 
To any continuous linear operator on $C^\infty(G)$ and then to any pseudo-differential operator $A$ in the class $\Psi^m_{\rho,\delta}(G):=\Psi^m_{\rho,\delta}(G\times \widehat{G}),$ $0\leq \delta\leq \rho\leq 1,$ one can associate a matrix-valued global symbol $$ a:G\times \widehat{G}\rightarrow\bigcup_{[\xi]\in \widehat{G} }\mathbb{C}^{d_\xi\times d_\xi}, \,\,(x,[\xi])\mapsto a(x,[\xi])\in \mathbb{C}^{d_\xi\times d_\xi}, $$ allowing the global quantisation formula
\begin{equation}
    Af(x)=\sum_{[\xi]\in \widehat{G}}\smallint\limits_{G}d_\xi \textnormal{Tr}[\xi(y^{-1}x)a(x,\xi)]f(y)\textnormal{d}y,\,\,\forall f\in C^\infty(G),\,\forall x\in G.
\end{equation}
The problem of finding criteria  to assure the membership of a pseudo-differential operator $A$ in the Schatten classes $\mathscr{S}_{r}(L^2(G)),$ in terms of its  matrix-valued symbol $a:=a(x,[\xi])$  has been a source of intensive mathematical activity for around 10 years, see e.g. \cite{Delgado2013,DelgadoRuzhansky2014,DelRuzTrace1111,DelRuzTrace111,Delgado2015,DelgadoRuzhanskyWang,DelgadoRuzhanskyWang2,DelgadoRuzhanskyTokmagambetov2017,DelRuzTrace11,DelgadoRuzhansky2018MS,DelRuzTrace1,ChatzakouDelgadoRuzhansky2021,DelgadoRuzhansky2021,Delgado2022,ChatzakouDelgadoRuzhansky2022,Chatzakou2022,CardonaKumar2021,CardonaDelCorral2020,CardonaDelCorral2020K,CardonaKumar2019,Cardona2019,CardonaDelgadoRuzhanksyLOcalWeyl,CardonaDelgadoRuzhansky2022}.  Contributing to the previous references, the main results of this work can be summarised in  Theorems \ref{CardonaRuzhanskyChatzakouToft2022Th1}, \ref{CardonaRuzhanskyChatzakouToft2022Th2} and \ref{CardonaRuzhanskyChatzakouToft2022Th3} below where we will use the following notations.
\begin{itemize}
\item We denote by $ \mathfrak{g}$   the Lie algebra  of a compact Lie group $G.$ The mapping  $B(X,Y)=-\textnormal{Tr}[\textnormal{ad}(X)\textnormal{ad}(Y)],$  $X,Y\in\mathfrak{g}, $ is the  Killing form on $\mathfrak{g}\times \mathfrak{g}$  and we denote by  $|| X ||_{g}:=\sqrt{-B(X,X)}$ the corresponding norm on $\mathfrak{g}$ associated to $-B.$\footnote{Since $G$ is a compact Lie group the positive Killing form  $-B:\mathfrak{g}\times \mathfrak{g}\rightarrow \mathbb{C},$ is positive semi-definite, i.e. $-B(X,X)>0,$ $\forall X\in \mathfrak{g}\setminus \{0\}.$}
    \item We denote by $\mathcal{L}_G$  the positive Laplace Beltrami operator on $G,$ and  under the identification $\mathfrak{g}^*\cong \mathfrak{g}$, $\eta\mapsto ||\eta ||^2_g,$ $\eta\in \mathfrak{g}^*\setminus \{0\},  $ denotes its principal symbol.
    
    \item The family $\mathscr{S}_r(L^2(G)),$ $0<r<\infty,$ is formed by the Schatten von Neumann ideals on a compact Lie group $G,$ if $0<r<\infty,$ and for $r=\infty,$ $\mathscr{S}_r(L^2(G))=\mathscr{B}(L^2(G))$ denotes the algebra of all bounded linear operators  on $L^2(G).$
    \item For every $r\in (0,\infty),$ the Schatten norm of a symbol $a(x,[\xi])$ is given by $\Vert a(x,[\xi])\Vert_{\mathscr{S}_r}=\textnormal{Tr}[|a(x,[\xi])|^{r}]^{\frac{1}{r}},$ where $|a(x,[\xi])|:=\sqrt{a(x,[\xi])^*a(x,[\xi])}$ is defined in terms of the functional calculus of matrices. Note that $\Vert a(x,[\xi])\Vert_{\mathscr{S}_2}= \Vert a(x,[\xi])\Vert_{\textnormal{HS}}$ is the standard Hilbert-Schmidt norm of matrices.
    \item For any $1\leq p_1,p_2<\infty,$ the space $L^{p_1}(G,\mathscr{S}_{p_2}(\widehat{G}))$ is defined by those symbols $a:=a(x,[\xi])$ such that 
    $$  \Vert a(\cdot ,\cdot)\Vert_{L^{p_1}(G,\mathscr{S}_{p_2}(\widehat{G}))}=\left(\smallint\limits_G\Vert a(x,\cdot)\Vert_{\mathscr{S}_{p_2}(\widehat{G})}^{p_1}\textnormal{d}x\right)^{\frac{1}{p_1}}<\infty,$$ 
    where $$  \Vert a(x,\cdot)\Vert_{\mathscr{S}_{p_2}(\widehat{G})}= \left(\sum_{[\xi]\in \widehat{G}}d_{\xi}\Vert a(x,[\xi])\Vert_{\mathscr{S}_{p_2}}^{p_2} \right)^{\frac{1}{p_2}}. $$
    \item In terms of the $\ell^p(\widehat{G})$ norm
    \begin{equation}
        \Vert a(x,[\xi])\Vert_{\ell^p(\widehat{G})}=\left(\sum_{[\xi]\in \widehat{G}}d_{\xi}^{p\left(\frac{2}{p}-\frac{1}{2}\right)}\Vert a(x,[\xi])\Vert_{\textnormal{HS}}^{p} \right)^{\frac{1}{p}},\,\,1\leq p<\infty,
    \end{equation} for any $1\leq p_1,p_2<\infty,$ the space $L^{p_1}(G,\ell^{p_2}(\widehat{G}))$ is defined by those symbols $a:=a(x,[\xi])$ such that 
    $$  \Vert a(\cdot,\cdot)\Vert_{L^{p_1}(G,\ell^{p_2}(\widehat{G}))}=\left(\smallint\limits_G\Vert a(x,\cdot)\Vert_{\ell^{p_2}(\widehat{G})}^{p_1}\textnormal{d}x\right)^{\frac{1}{p_1}}<\infty,$$
\end{itemize}we refer the reader to \cite{FischerRuzhanskyBook} for the embeddings between these two classes of spaces. The following three theorems summarise our main results. We start with our characterisation of elliptic operators in Schatten classes.
\begin{theorem}[General symbols]\label{CardonaRuzhanskyChatzakouToft2022Th1} Let $G$ be a compact Lie group of dimension $n,$ let $m\in \mathbb{R},$ $r>0, $ and let $0\leq \delta<\rho\leq 1.$  Consider an elliptic pseudo-differential operator $A\in \Psi^m_{\rho,\delta}(G\times \widehat{G})$. The following conditions are equivalent:
    \begin{itemize}
    \item[(1)] $m<0$ and  $A$ belongs to the Schatten class of order $r>0,$ that is $A\in \mathscr{S}_r(L^2(G));$
    \item[(2)] The Bessel potential of order $m$ belongs to the Schatten class of order $r>0:$ $B_{m}:=(1+\mathcal{L}_G)^{\frac{m}{2}}\in \mathscr{S}_r(L^2(G)). $
    \item[(3)] The matrix-valued symbol of $|A|^{\frac{r}{2}}$ satisfies the following summability condition
    \begin{equation}
        \sum_{[\xi]\in \widehat{G}}d_\xi\smallint\limits_G\|\sigma_{|A|^{\frac{r}{2}}}(x,[\xi])\|_{\textnormal{HS} }^2\textnormal{d}x<\infty;
    \end{equation}
    \item[(4)] The matrix-valued symbol of $A$ satisfies the following summability condition
    \begin{equation}
        \smallint\limits_G\sum_{[\xi]\in \widehat{G}}d_\xi\|\sigma_{A}(x,[\xi])\|_{\mathscr{S}_r }^r\textnormal{d}x<\infty;
    \end{equation}
    \item[(5)] $m<-n/r.$
\end{itemize} 
     Moreover, if $A\in \Psi^m_{\rho,\delta}(G\times \widehat{G})$ is not elliptic and $m<-n/r,$ then we have that  $A\in \mathscr{S}_{r}(L^2(G)).$
\end{theorem}
As for classical operators on compact Lie groups we have the following result.
\begin{theorem}[Classical symbols]\label{CardonaRuzhanskyChatzakouToft2022Th2}
Let $G$ be a compact Lie group of dimension $n,$ let $m\in \mathbb{R},$ $r>0, $ and let $0\leq \delta<\rho\leq 1.$    Let $A\in \Psi^{m}_{1,0}(G)$ be a classical  pseudo-differential operator of order $m$.
\begin{itemize}
    \item[(6)] If $r\in [1,\infty)\cap \mathbb{Z},$ then
$A\in \mathscr{S}_r(L^2(G))$ if and only if $m<-n/r.$ For $r=\infty,$     $A\in \mathscr{B}(L^2(G))$ if and only if $m\leq 0.$ 
\item[(7)] If $r\in (1,\infty)\setminus \mathbb{Z},$ and   $A\in \mathscr{S}_{r}(L^2(G)),$ then $m\leq -n/r.$ Moreover, if $A$ is elliptic then one has the strict inequality $m<-n/r.$
\end{itemize}
Additionally, consider the subclass $\Psi_0^m(G)$ in  $\Psi_{cl}^m(G)$  of
operators with homogeneous symbols of
order $m.$ Then
the following conditions are equivalent:
\begin{itemize}
\item[(8)] $\Psi ^m_{cl}(G)\subseteq \mathscr{S}_r(L^2(G)),$ $r\in (0,\infty]$;\

\item[(9)] $\Psi ^m_0(G)\subseteq \mathscr{S}_r(L^2(G)),$ $r\in (0,\infty]$.
\end{itemize} 
\end{theorem}
In the next theorem we consider conditions of limited regularity.
\begin{theorem}[Symbols of low regularity]\label{CardonaRuzhanskyChatzakouToft2022Th3}
Let $G$ be a compact Lie group of dimension $n.$   Let us assume that for any $[\xi],$ the symbol $a(\cdot,[\xi])$ is Haar measurable. Then:
\begin{itemize}
    \item[(10)] Assume that a symbol $a\in L^p(G,\ell^p(\widehat{G}))$ for some $1<p<2.$ Then, the corresponding pseudo-differential operator satisfies $A\in \mathscr{S}_{p'}(L^2(G)),$ where $p'=p/(p-1).$
    \item[(11)] Assume that the matrix-valued symbol $a=a(x,[\xi])$ satisfies the regularity condition
\begin{equation}
    \Vert (1+\mathcal{L}_G)^{\frac{N}{2}}\sigma_A(x,\cdot)\Vert_{L^1(G,\mathscr{S}_p(\widehat{G}))}=\smallint\limits_G\Vert (1+\mathcal{L}_G)^{\frac{N}{2}}\sigma_A(x,\cdot)\Vert_{\mathscr{S}_p(\widehat{G})}\textnormal{d}x<\infty
\end{equation}where $N>n.$ Then $A\in \mathscr{S}_{p}(L^2(G))$ provided that $1\leq p<\infty.$ 
\end{itemize}
\end{theorem}

\subsection{Open problems}
In view of the open questions that have arisen in our approach, and based on the active research on this field in the last 10 years, we propose the following open problems.
\begin{openproblem} Let $0\leq \delta<\rho\leq 1.$ Prove (or disprove) that if  $A\in \Psi^m_{\rho,\delta}(G\times \widehat{G})$ is a non-elliptic pseudo-differential operator that belongs to the Schatten class  $ \mathscr{S}_{r}(L^2(G))$ then $m<-n/r.$
\end{openproblem}
\begin{openproblem}Prove (or disprove) that for  $r\in (1,\infty)\setminus \mathbb{Z},$ and with  $A\in \mathscr{S}_{r}(L^2(G))$ being a non-elliptic operator, then the inequality $m\leq -n/r$ in  $\textnormal{(7)}$ of Theorem \ref{CardonaRuzhanskyChatzakouToft2022Th2} can be improved to the strict order estimate $m<-n/r.$    
\end{openproblem}
\begin{remark}We observe that the equivalence $\textnormal{(1)}\Longleftrightarrow \textnormal{(3)} $ in Theorem \ref{CardonaRuzhanskyChatzakouToft2022Th1} was proved by the third author and J. Delgado in \cite{DelgadoRuzhansky2021}.     In particular, in \cite{DelgadoRuzhansky2014} the relation between the spectral trace and the nuclear trace of operators has been investigated for the more general notion of nuclear operators and Grothendieck-Lidskii type formulas.  Further analysis involving criteria in terms of matrix-valued symbols  was also carried out  on compact Lie groups and on arbitrary compact manifolds in  \cite{CardonaRuzhanksySubellipticCalculus,DelgadoRuzhansky2014,DelRuzTrace1111,DelRuzTrace111,DelgadoRuzhanskyWang,DelgadoRuzhanskyWang2,DelgadoRuzhanskyTokmagambetov2017,DelRuzTrace11,DelgadoRuzhansky2018MS,DelRuzTrace1,DelgadoRuzhansky2021}.
\end{remark}
\subsection{Organisation of the manuscript}
This paper is organised as follows.
\begin{itemize}
    \item In Section \ref{Preliminaries} we present the basics of the Fourier analysis on compact Lie groups used here as well as the preliminaries about the pseudo-differential calculus on compact Lie groups in terms of the matrix-valued symbols as developed in \cite{Ruz}. 
    \item Section \ref{Schatten:Properties} will be dedicated to the proof of our main Theorems. More precisely:
\begin{itemize}
    \item[1.] Theorem \ref{CardonaRuzhanskyChatzakouToft2022Th1} is presented later  as Theorem \ref{Schatten:properties:}  of Subsection \ref{rhodeltasection}.
     \item[2.] (6) and (7) of Theorem \ref{CardonaRuzhanskyChatzakouToft2022Th2}  are  proved in Theorem \ref{Invariance:trace2r}  of Subsection \ref{ClassicalSection}.
     \item[3.] The equivalence $\textnormal{(8)}\Longleftrightarrow \textnormal{(9)} $ of   Theorem \ref{CardonaRuzhanskyChatzakouToft2022Th2} is proved in Proposition \ref{Prop:RedToHom} of Subsection \ref{Classical:symbols:section}.
     \item[4.]   Theorem \ref{CardonaRuzhanskyChatzakouToft2022Th3} is proved in Subsection \ref{limited:regularity:section} (see Theorem \ref{Michael:Criterion} and Theorem \ref{Michael:criterion:2}, respectively). 
\end{itemize}
\item Finally, in Subsection \ref{Last:Subsection} we prove that the order of a matrix-valued symbol associated to a classical pseudo-differential operator classifies its operator order. For the proof we use the approach developed in this work for the analysis of Schatten operators that involves the average of their principal symbols  on the co-sphere (see Figure \ref{Average}).
\end{itemize}

\section{Preliminaries}\label{Preliminaries}

\subsection{The Fourier analysis of a compact Lie group} Let $\textnormal{d}x$ be the Haar measure on a compact Lie group $G.$  
  The Hilbert space $L^{2}(G):=L^2(G,\textnormal{d}x)$ will be endowed with
   the inner product $$ (f,g)=\smallint\limits_{G}f(x)\overline{g(x)}\textnormal{d}x.$$   The Peter-Weyl theorem gives a spectral decomposition of $L^2(G)$  in terms of the entries of unitary representations of  $G$. In order to present such a result we will give some preliminaries.
\begin{definition}[Unitary representation of $G$]
    A continuous and unitary representation of  $G$ on $\mathbb{C}^{\ell}$ is any continuous mapping $\xi\in\textnormal{Hom}(G,\textnormal{U}(\ell)) ,$ where $\textnormal{U}(\ell)$ is the Lie group of unitary matrices of order $\ell\times \ell.$ The integer number $\ell=d_{\xi}$ is called the dimension of the representation $\xi.$ 
\end{definition}

\begin{remark}[Irreducible representations]  We recall that:
\begin{itemize} 
    \item  a subspace $L\subseteq \mathbb{C}^{d_\xi}$ is called $\xi$-invariant if for any $x\in G,$ $\xi(x)(L)\subseteq L,$ where $\xi(x)(L):=\{\xi(x)v:v\in L\}.$
    \item The representation $\xi$ is irreducible if its only invariant subspaces are $L=\emptyset$ and $L=\mathbb{C}^{d_\xi},$ the trivial ones.
    \item Any unitary representation $\xi$ is a direct sum of unitary irreducible representations. We denote it by $\xi=\xi_1\otimes \cdots\otimes \xi_k,$ with $\xi_i,$ $1\leq i\leq k,$ being irreducible representations on factors $\mathbb{C}^{d_{\xi_i}}$ that decompose the representation space $$ \mathbb{C}^{d_{\xi}}=\mathbb{C}^{d_{\xi_1}}\otimes \cdots \otimes \mathbb{C}^{d_{\xi_k}} .$$    
\end{itemize}  The notion of {\it equivalent representations} allows us to define an equivalence relation in the family of unitary representations. We recall it in the following definition.
\end{remark}
\begin{definition}[Equivalent representations]
    Two unitary representations $$ \xi\in \textnormal{Hom}(G,\textnormal{U}(d_\xi)) \textnormal{ and  }\eta\in \textnormal{Hom}(G,\textnormal{U}(d_\eta))$$  are equivalent if there exists a linear mapping $S:\mathbb{C}^{d_\xi}\rightarrow \mathbb{C}^{d_\eta}$ such that for any $x\in G,$ $S\xi(x)=\eta(x)S.$ The mapping $S$ is called an intertwining operator between $\xi$ and $\eta.$ The set of all the intertwining operators between $\xi$ and $\eta$ is denoted by $\textnormal{Hom}(\xi,\eta).$
\end{definition}
\begin{remark}[Schur Lemma, 1905]
If $\xi\in \textnormal{Hom}(G,\textnormal{U}(d_\xi)) $ is irreducible, then $\textnormal{Hom}(\xi,\xi)=\mathbb{C}I_{d_\xi}$ is formed by scalar multiples of the identity matrix  $I_{d_\xi}$ of order $d_\xi.$

\end{remark}
\begin{definition}[The unitary dual]
    The relation $\sim$ on the set of unitary representations, which we denote by $\textnormal{Rep}(G),$ and defined by: {\it $\xi\sim \eta$ if and only if $\xi$ and $\eta$ are equivalent representations,} is an equivalence relation. The quotient set
$$
    \widehat{G}:={\textnormal{Rep}(G)}/{\sim}
$$is called the unitary dual of $G.$ Since $G$ is a compact Lie group, $\widehat{G}$ is a discrete set.
\end{definition}
The unitary dual encodes all the Fourier analysis on the group. The Fourier transform is defined as follows.
\begin{definition}[Group Fourier transform]
    If $\xi\in \textnormal{Rep}(G),$ the Fourier transform $\mathscr{F}_{G}$ associates to any $f\in C^\infty(G)$ a matrix-valued function $\mathscr{F}_{G}f$ defined on $\textnormal{Rep}(G)$ as follows
$$ (\mathscr{F}_{G}f)(\xi) \equiv   \widehat{f}(\xi)=\int\limits_Gf(x)\xi(x)^{*}\textnormal{d}x,\,\,\xi\in \textnormal{Rep}(G). $$ 
\end{definition}
\begin{remark}[The Fourier inversion formula on a compact Lie group]
The discrete Schwartz space $\mathscr{S}(\widehat{G}):=\mathscr{F}_{G}(C^\infty(G))$ is the image of the Fourier transform on the class of smooth functions. This operator admits a unitary extension from $L^2(G)$ into $\ell^2(\widehat{G}),$ with 
\begin{equation}
 \ell^2(\widehat{G})=\left\{\phi:\forall [\xi]\in \widehat{G},\,\phi(\xi)\in \mathbb{C}^{d_\xi\times d_\xi}\textnormal{ and }\Vert \phi\Vert_{\ell^2(\widehat{G})}<\infty \right\},  
\end{equation} where
$$  \Vert \phi\Vert_{\ell^2(\widehat{G})}:=\left(\sum_{[\xi]\in \widehat{G}}d_{\xi}\Vert\phi(\xi)\Vert_{\textnormal{HS}}^2\right)^{\frac{1}{2}}. $$
The norm $\Vert\phi(\xi)\Vert_{\textnormal{HS}}=(\textnormal{Tr}(\phi(\xi)^*\phi(\xi)))^{\frac{1}{2}}$ is the standard Hilbert-Schmidt norm of matrices. The Fourier inversion formula takes the form
\begin{equation}
    f(x)=\sum_{[\xi]\in \widehat{G}} d_{\xi}\textnormal{Tr}[\xi(x)\widehat{f}(\xi)],\forall f\in L^1(G), \forall x\in G,
\end{equation}where the summation is understood in the sense that from any equivalence class $[\xi]$ we choose one (any)  unitary representation.      
\end{remark}

\subsection{The quantisation formula} Let  $A:C^\infty(G)\rightarrow C^\infty(G)$ be a continuous linear operator with respect to the standard Fr\'echet structure on $C^\infty(G).$ There is a way of associating to the operator $A$ a matrix-valued function $\sigma_A$ defined on the non-commutative phase space $G\times \widehat{G}$ to rewrite the operator $A$ in terms of the Fourier inversion formula and in terms of the Fourier transform. Such a expression is called the quantisation formula. To introduce it we require the following definition.
\begin{definition}[Right convolution kernel of an operator]
 The Schwartz kernel theorem associates to $A$ a kernel/distribution $K_A\in \mathscr{D}'(G\times G)$ such that
$$   Af(x)=\int\limits_{G}K_{A}(x,y)f(y)\textnormal{d}y,\,\,f\in C^\infty(G).$$ The distribution defined via $R_{A}(x,xy^{-1}):=K_A(x,y)$ that provides the convolution identity
$$   Af(x)=\int\limits_{G}R_{A}(x,xy^{-1})f(y)\textnormal{d}y,\,\,f\in C^\infty(G),$$
is called the right-convolution kernel of $A.$   
\end{definition}

\begin{remark}[The quantisation formula]
 Now, we will associate  a global symbol $\sigma_A:G\times \widehat{G}\rightarrow \cup_{\ell\in \mathbb{N}}\mathbb{C}^{\ell\times \ell}$ to $A.$ Indeed,  in view of the identity $Af(x)=(f\ast R_{A}(x,\cdot))(x),$  we get 
 $$ \widehat{Af}(\xi)= \widehat{R}_{A}(x,\xi)\widehat{f}(\xi). $$ Then we have
 that
 \begin{equation}\label{Quantisation:formula}
     Af(x)=\sum_{[\xi]\in \widehat{G}}d_\xi\textnormal{Tr}[\xi(x)\widehat{R}_{A}(x,\xi)\widehat{f}(\xi)],\,f\in C^\infty(G).
 \end{equation} In view of the identity \eqref{Quantisation:formula}, from any equivalence class $[\xi]\in \widehat{G},$ we can choose one and only one irreducible unitary representation $\xi_0\in [\xi],$ such that the matrix-valued function
 \begin{equation}
    \sigma_{A}(x,[\xi])\equiv \sigma_A(x,\xi_0):=\widehat{R}_{A}(x,\xi_0),\,(x,[\xi])\in G\times\widehat{G},
 \end{equation} satisfies that
\begin{equation}\label{Quantisation:formula2:}
     Af(x)=\sum_{[\xi]\in \widehat{G}}d_\xi\textnormal{Tr}[\xi_0(x)\sigma_{A}(x,[\xi])\widehat{f}(\xi_0)],\,f\in C^\infty(G).
 \end{equation}
 The representation in \eqref{Quantisation:formula2:} is independent of the choice of the representation $\xi_0\in \textnormal{Rep}(G)$ from any equivalent class $[\xi]\in \widehat{G}.$ This is a consequence of the Fourier inversion formula.   
 So, we can simply write
 \begin{equation}\label{Quantisation:formula2}
     Af(x)=\sum_{[\xi]\in \widehat{G}}d_\xi\textnormal{Tr}[\xi(x)\sigma_{A}(x,[\xi])\widehat{f}(\xi)],\,\forall f\in C^\infty(G).
 \end{equation}  
\end{remark}
In the following quantisation theorem we observe that the distribution $\sigma_A$ in \eqref{Quantisation:formula2} defined on $G\times \widehat{G}$ is unique and can be written in terms of the operator $A,$ see Theorems  10.4.4 and 10.4.6 of \cite{Ruz}.
\begin{theorem}\label{The:quantisation:thm}
    Let $A:C^\infty(G)\rightarrow C^\infty(G) $ be a continuous linear operator. The following statements are equivalent.
    \begin{itemize}
        \item The matrix-valued distribution $\sigma_A(x,[\xi]):G\times \widehat{G}\rightarrow \cup_{\ell\in \mathbb{N}}\mathbb{C}^{\ell\times \ell}$ satisfies that
        \begin{equation}\label{Quantisation:3}
            \forall f\in C^\infty(G),\,\forall x\in G,\,\, Af(x)=\sum_{[\xi]\in \widehat{G}}d_\xi\textnormal{Tr}[\xi(x)\sigma_{A}(x,[\xi])\widehat{f}(\xi)].
        \end{equation}
        \item We have that $ 
            \forall (x,[\xi])\in G\times \widehat{G},\, \sigma_{A}(x,\xi)=\widehat{R}_A(x,\xi).
        $ \\
        \item  The following identity holds: $ 
          \forall (x,[\xi])\in G\times \widehat{G},\, \sigma_A(x,\xi)=\xi(x)^{*}A\xi(x),$ where $ A\xi(x):=[A\xi_{ij}(x)]_{i,j=1}^{d_\xi}.  
        $ 
    \end{itemize}
\end{theorem}
\begin{remark}
    In view of the quantisations formulae \eqref{Quantisation:formula2} and \eqref{Quantisation:3}, a symbol $\sigma_A$ can be considered as a mapping defined on $G\times \widehat{G}$ or as a mapping  defined on $ G\times \textnormal{Rep}(G)$ by identifying all the values $\sigma_A(x,\xi)=\sigma_A(x,\xi')=\sigma(x,[\xi])$ when $\xi',\xi\in [\xi].$
\end{remark}
\begin{example}[The symbol of a measurable function of the Laplacian] Let $\mathbb{X}=\{X_1,\cdots,X_n\}$ be an orthonormal basis of the Lie algebra $\mathfrak{g}.$ The positive Laplacian on $G$ is the second order differential operator 
\begin{equation}
    \mathcal{L}_G=-\sum_{j=1}^nX_j^2.
\end{equation}The operator $ \mathcal{L}_G$ is independent of the choice of the orthonormal basis $\mathbb{X}$ of $\mathfrak{g},$ see e.g. \cite{Ruz}. The $L^2$-spectrum of $\mathcal{L}_G$ is a discrete set that can be enumerated in terms of the unitary dual $\widehat{G}$ as
\begin{equation}
    \textnormal{Spectrum}(\mathcal{L}_G)=\{\lambda_{[\xi]}:[\xi]\in \widehat{G}\}.
\end{equation}For a Borel function $f:\mathbb{R}^+_0\rightarrow \mathbb{C},$ the right-convolution kernel $R_{f(\mathcal{L}_G)}$ of the operator $f(\mathcal{L}_G)$ (defined by the spectral calculus) is determined by the identity
\begin{equation}
    f(\mathcal{L}_G)\phi(x)=\phi\ast R_{f(\mathcal{L}_G)}(x),\,x\in G,
\end{equation}where
\begin{equation}
  \forall [\xi]\in \widehat{G},\,\,  \widehat{R}_{f(\mathcal{L}_G)}([\xi])=f(\lambda_{[\xi]})I_{d_\xi}.
\end{equation}Then the matrix-valued  symbol of $f(\mathcal{L}_G)$ can be determined e.g. using  Theorem \ref{The:quantisation:thm} as follows
\begin{equation}
    \sigma_{f(\mathcal{L}_G)}(x,\xi)=\widehat{R}_{f(\mathcal{L}_G)}([\xi]).
\end{equation}Since the operator $f(\mathcal{L}_G)$ is left-invariant the symbol  $\sigma_{f(\mathcal{L}_G)}(\xi)=\sigma_{f(\mathcal{L}_G)}(x,\xi)$ does not depend of  $x\in G.$ Of particular interest for the definition of the global H\"ormander classes on $G,$  will be the Japanese bracket function
\begin{equation}
    \langle t\rangle:=(1+t)^{\frac{1}{2}},\,t\geq -1.
\end{equation}In particular the symbol of the operator $ \langle \mathcal{L}_G\rangle=(1+\mathcal{L}_G)^{\frac{1}{2}}$ is given by
\begin{equation}\label{Japanne:bracket:G}
     \sigma_{\langle \mathcal{L}_G\rangle}([\xi]):= \langle \xi \rangle I_{d_\xi}, \,\,\,\langle \xi \rangle:=\langle \lambda_{[\xi]} \rangle.  
\end{equation}

\end{example}
\subsection{Global H\"ormander classes on compact Lie groups} In this section we denote for any linear mapping $U$ on $\mathbb{C}^n$ by $\Vert U\Vert_{\textnormal{op}}$ the standard operator norm
$$ \Vert U\Vert_{\textnormal{op}}= \Vert U\Vert_{\textnormal{End}(\mathbb{C}^n)}:=\sup_{l\neq 0}\|Ul\|_{e}/\|l\|_{e} , $$ where $\|l\|_{e}=(l_1^2+\cdots +l_n^2)^{\frac{1}{2}} $ is the Euclidean norm.

For introducing the H\"ormander classes on compact Lie groups we have to measure the growth of derivatives of symbols in the group variable, for this we use vector fields $X\in T(G).$ To derivate symbols with respect to the discrete variable $[\xi]\in \widehat{G}$ we use difference operators. Before introducing the H\"ormander classes on compact Lie groups we have to define these differential/difference operators. 

\begin{definition}[Left-invariant canonical differential operators] If $\{X_{1},\cdots, X_{n}\}$ is an arbitrary family of left-invariant vector fields, we will denote by
$$  X_{x}^{\alpha}:=X_{1,x}^{\alpha_1}\cdots X_{n,x}^{\alpha_n} $$
an arbitrary canonical differential operator of order $m=|\alpha|.$    
\end{definition}

Also, we have to take derivatives with respect to the ``discrete'' frequency variable $\xi\in \textnormal{Rep}(G).$ To do this, we will use the notion of difference operators introduced in \cite{RuzhanskyWirth2015}. Indeed,  the frequency variable in the symbol $\sigma_A(x,[\xi])$ of a continuous and linear operator $A$ on $C^\infty(G)$ is discrete. This is since $\widehat{G}$ is a discrete space.

\begin{definition}[Canonical difference operators $\mathbb{D}^\alpha$ on the dual $\widehat{G}$] If $\xi_{1},\xi_2,\cdots, \xi_{k},$ are  fixed irreducible and unitary  representations of $G$, which not necessarily belong to the same equivalence class, then each coefficient of the matrix
\begin{equation}
 \xi_{\ell}(g)-I_{d_{\xi_{\ell}}}=[\xi_{\ell}(g)_{ij}-\delta_{ij}]_{i,j=1}^{d_{\xi_\ell}},\, \quad g\in G, \,\,1\leq \ell\leq k,
\end{equation} 
that is each function 
$q^{\ell}_{ij}(g):=\xi_{\ell}(g)_{ij}-\delta_{ij}$, $ g\in G,$ defines a difference operator
\begin{equation}\label{Difference:op:rep}
    \mathbb{D}_{\xi_\ell,i,j}:=\mathscr{F}_G(\xi_{\ell}(g)_{ij}-\delta_{ij})\mathscr{F}^{-1}_G.
\end{equation}
We can fix $k\geq \mathrm{dim}(G)$ of these representations in such a way that the corresponding  family of difference operators is admissible, that is, 
\begin{equation*}
    \textnormal{rank}\{\nabla q^{\ell}_{i,j}(e):1\leqslant \ell\leqslant k \}=\textnormal{dim}(G).
\end{equation*}
To define higher order difference operators of this kind, let us fix a unitary irreducible representation $\xi_\ell$.
Since the representation is fixed we omit the index $\ell$ of the representations $\xi_\ell$ in the notation that will follow.
Then, for any given multi-index $\alpha\in \mathbb{N}_0^{d_{\xi_\ell}^2}$, with 
$|\alpha|=\sum_{i,j=1}^{d_{\xi_\ell}}\alpha_{i,j}$, we write
$$\mathbb{D}^{\alpha}:=\mathbb{D}_{1,1}^{\alpha_{11}}\cdots \mathbb{D}^{\alpha_{d_{\xi_\ell},d_{\xi_\ell}}}_{d_{\xi_\ell}d_{\xi_\ell}}
$$ 
for a difference operator of order $m=|\alpha|$.    
\end{definition}
Now, we are rea\textnormal{d}y for introducing the global H\"ormander classes on compact Lie groups.
\begin{definition}[Global $(\rho,\delta)$-H\"ormander classes in the whole range $0\leq \delta,\rho\leq 1$]
    We say that $\sigma \in {S}^{m}_{\rho,\delta}(G\times \widehat{G})$ if the following symbol inequalities 
\begin{equation}\label{HormanderSymbolMatrix}
   \Vert {X}^\beta_x \mathbb{D}^{\alpha} \sigma(x,\xi)\Vert_{\textnormal{op}}\leqslant C_{\alpha,\beta}
    \langle \xi \rangle^{m-\rho|\gamma|+\delta|\beta|},
\end{equation} are satisfied for all multi-indices $\beta$ and  $\gamma ,$  and for all $(x,[\xi])\in G\times \widehat{G},$ where $ \langle \xi \rangle$ denotes the Japanese bracket function at $\lambda_{[\xi]}\in \textnormal{Spectrum}[\mathcal{L}_G]
$ defined in \ref{Japanne:bracket:G}.
\end{definition}
The class $\Psi^m_{\rho,\delta}(G\times \widehat{G})\equiv\textnormal{Op}({S}^m_{\rho,\delta}(G\times \widehat{G}))$ is defined by those continuous and linear operators on $C^\infty(G)$ such that $\sigma_A\in {S}^m_{\rho,\delta}(G\times \widehat{G}).$ 

In the next theorem we describe some fundamental properties of the global H\"ormander  classes of pseudo-differential operators  \cite{Ruz}.
\begin{theorem}\label{RTcalculus:Group} Let $\rho,\delta\in [0,1]$ be such that $0\leqslant \delta\leqslant \rho\leqslant 1,$ $\rho\neq 1.$  Then  $$\Psi^\infty_{\rho,\delta}(G):=\cup_{m\in \mathbb{R}} \Psi^m_{\rho,\delta}(G) $$   is an algebra of operators stable under compositions and adjoints, that is:
\begin{itemize}
    \item [1.] the mapping $$A\mapsto A^{*}:\Psi^{m}_{\rho,\delta}(G\times \widehat{G})\rightarrow \Psi^{m}_{\rho,\delta}(G\times \widehat{G})$$  is a continuous linear mapping between Fr\'echet spaces. 
\item [2.] The mapping $$ (A_1,A_2)\mapsto A_1\circ A_2: \Psi^{m_1}_{\rho,\delta}(G\times \widehat{G})\times \Psi^{m_2}_{\rho,\delta}(G\times \widehat{G})\rightarrow \Psi^{m_1+m_2}_{\rho,\delta}(G\times \widehat{G})$$ is a continuous bilinear mapping between Fr\'echet spaces.
\end{itemize}Moreover, any operator in the class   $ \Psi^{0}_{\rho,\delta}(G\times \widehat{G})$ admits a  bounded extension from $L^2(G)$ to  $L^2(G).$
\end{theorem} 
\begin{remark}With $0\leqslant \delta<\rho\leqslant 1$ such that $\rho\geq 1-\delta,$ the condition $A\in \Psi^m_{\rho,\delta}(G\times \widehat{G})$ where $m\in \mathbb{R},$ is equivalent to the fact that, when microlocalising the operator $A$ into a local coordinate system $U,$ the operator $A$ takes the form
\begin{equation*}
   Af(x)=\smallint\limits_{\mathbb{R}^n}\smallint\limits_{\mathbb{R}^n}e^{2\pi i (x-y)\cdot \xi}a(x,\xi)f(y)\textnormal{d}y\textnormal{d}\xi,\,\, \forall f\in C^\infty_0(U),\,\forall x\in \mathbb{R}^n,
\end{equation*} 
where the function $a=a_U,$ is  such that  for every compact subset $K\subseteq U$ and for all $\alpha,\beta\in \mathbb{N}_0^n,$  the inequalities\begin{equation}\label{seminorms}
  |\partial_{x}^\beta\partial_{\xi}^\alpha a(x,\xi)|\leqslant C_{\alpha,\beta,K}(1+|\xi|)^{m-\rho|\alpha|+\delta|\beta|},
\end{equation} hold uniformly in $(x,\xi)\in K\times \mathbb{R}^n.$  This characterisation of the H\"ormander classes on $G$  was proved in \cite{RuzhanskyTurunenWirth2014}. So, for any compact Lie group $G,$ the classes $\Psi^m_{\rho,\delta}(G\times \widehat{G})$ agree with the ones introduced by H\"ormander \cite{Hormander1985III} when $0\leqslant \delta<\rho\leqslant 1$ and  $\rho\geq 1-\delta.$   
\end{remark}

\section{Schatten Properties}\label{Schatten:Properties}
In this section we analyse the membership of pseudo-differential operators in the Schatten classes on $L^2(G).$

\subsection{Schatten properties for operators. Limited regularity symbols}\label{limited:regularity:section}
In this section we stu\textnormal{d}y the Schatten properties of operators with symbols of limited regularity. 
Without assumptions of regularity we start with the following criterion.

\begin{theorem}\label{Michael:Criterion}
    Assume that for a symbol $\sigma_A$ we have  $\sigma_A\in L^p(G,\ell^p(\widehat{G}))$ for some $1<p<2.$ Then, the corresponding pseudo-differential operator $A$ satisfies that $A\in \mathscr{S}_{p'}(L^2(G)),$ where $p'=p/(p-1).$
\end{theorem}
\begin{proof}
    Let us consider the following criterion due to Russo (see \cite{Russo77}):
    \begin{equation}
        \Vert A\Vert_{\mathscr{S}_{p'}}\leq (\Vert K \Vert_{p,p'}\times \Vert K^{*} \Vert_{p,p'})^{1/2},
    \end{equation}where $K$ is the kernel of $A$ and $K^*$ is the kernel of the adjoint operator $A^*,$ $1<p<2,$ and 
\begin{equation}
    \Vert K\Vert_{p,p'}=\left(\smallint\limits_{G} \left( \smallint\limits_{G}|K(x,y)|^p\textnormal{d}x      \right)^{\frac{p'}{p}}\textnormal{d}y\right)^{\frac{1}{p'}}.
\end{equation}   
For a moment, let us assume that $A$ is self-adjoint. Then, $K=K^*$ and then
\begin{equation}
        \Vert A\Vert_{\mathscr{S}_{p'}}\leq \Vert K \Vert_{p,p'}.
    \end{equation}
Observe that the Hausdorff-Young inequality (see e.g. \cite[Page 69]{FischerRuzhanskyBook}) gives
\begin{align*} \smallint\limits_{G}|K(x,y)|^{p'}\textnormal{d}y &=\Vert y\mapsto \mathscr{F}_{G}^{-1}(\sigma_A(x,\cdot))(xy^{-1})\Vert_{L^{p'}(G;\textnormal{d}y)}^{p'}\\
&=\Vert z\mapsto \mathscr{F}_{G}^{-1}(\sigma_A(x,\cdot))(z)\Vert_{L^{p'}(G;\textnormal{d}z)}^{p'}\leq \Vert \sigma_A(x,\cdot)\Vert_{\ell^{p}(\widehat{G})}^{p'}.     
\end{align*}
Since $p'>2>p$ the Minkowski integral inequality implies that
\begin{align*}
    \Vert K\Vert_{p,p'} &\leq \Vert K\Vert_{p',p}=\left(\smallint\limits_{G} \left( \smallint\limits_{G}|K(x,y)|^{p'}\textnormal{d}y      \right)^{\frac{p}{p'}}\textnormal{d}x\right)^{\frac{1}{p}}\leq \left(\smallint\limits_{G}      \Vert \sigma_A(x,\cdot)\Vert_{\ell^{p}(\widehat{G})}^{p}\textnormal{d}x\right)^{\frac{1}{p}}\\
    &=\Vert\sigma_A\Vert_{L^p(G,\ell^{p}(\widehat{G}))}.
\end{align*} So, we have proved the statement in  Theorem \ref{Michael:Criterion} if $A$ is self-adjoint. Now, in the general case consider the decomposition of $A$ into its real and imaginary part
\begin{equation}
    A=\textnormal{Re}(A)+i\textnormal{Im}(A).
\end{equation}Note that
\begin{equation}
  \forall(x,[\xi])\in G\times \widehat{G},\,  \sigma_A(x,[\xi])=\sigma_{\textnormal{Re}(A)}(x,[\xi])+i\sigma_{\textnormal{Im}(A)}(x,[\xi]),
\end{equation}
that 
\begin{equation}
\forall  x\in G,\,  \|\sigma_A(x,[\xi])\|_{ \ell^{p}(\widehat{G}) }\asymp \|\sigma_{\textnormal{Re}(A)}(x,[\xi])\|_{ \ell^{p}(\widehat{G}) }+\|\sigma_{\textnormal{Im}(A)}(x,[\xi])\|_{ \ell^{p}(\widehat{G}) },
\end{equation} and using that
\begin{equation}
    \| A\|_{\mathscr{S}_{p'}}\asymp  \| \textnormal{Re}(A)\|_{\mathscr{S}_{p'}}+ \| \textnormal{Im}(A)\|_{\mathscr{S}_{p'}},
\end{equation}and the following inequalities (in view of the self-adjointness of $\textnormal{Re}(A)$ and $\textnormal{Im}(A)$) we have that 
\begin{align*}
  \| A\|_{\mathscr{S}_{p'}} &\asymp  \| \textnormal{Re}(A)\|_{\mathscr{S}_{p'}}+ \| \textnormal{Im}(A)\|_{\mathscr{S}_{p'}}\\
  &\lesssim \|\sigma_{\textnormal{Re}(A)}(x,[\xi])\|_{L^p(G,\ell^{p}(\widehat{G})) }+\|\sigma_{\textnormal{Im}(A)}(x,[\xi])\|_{L^p(G,\ell^{p}(\widehat{G}))}\\
 & \asymp   \| \sigma_A\|_{L^p(G,\ell^{p}(\widehat{G})) }.
\end{align*}
The proof of Theorem \ref{Michael:Criterion} is complete.
\end{proof}
\begin{remark} In terms of the $\ell^p$-Schatten norm on $\widehat{G}$ one has the Hausdorff-Young inequality (see e.g. \cite[Page 67]{FischerRuzhanskyBook}) 
\begin{equation}\label{HY:in}
\Vert \widehat{f}\Vert_{\mathscr{S}_{p'}(\widehat{G})}\leq \Vert f\Vert_{L^p(G)}, \,\,1\leq  p\leq 2,    
\end{equation}
where $p'$ is the conjugate exponent of $p.$ However, since the group $G$ is compact one has the following refined versions for this inequality
\begin{equation}\label{HY:in:2}
    \Vert\mathscr{F}_G^{-1}\sigma\Vert_{L^{p'}(G)}\leq \Vert \sigma\Vert_{\ell^{p}(\widehat{G})},\,\,\Vert\widehat{f}\Vert_{\ell^{p'}(\widehat{G})}\leq \Vert f\Vert_{L^p(G)},\,1\leq p\leq 2.
\end{equation}In particular, the second inequality in \eqref{HY:in:2} is sharper than  \eqref{HY:in} because of the embedding $\mathscr{S}_{p'}(\widehat{G})\subset \ell^{p'}(\widehat{G})$ for $2\leq p'\leq \infty,$ see \cite[Page 70]{FischerRuzhanskyBook}. We have used the first inequality in \eqref{HY:in:2}, which is the Hausdorff-Young inequality for the inverse Fourier transform to estimate the inequality $$  \Vert K\Vert_{p,p'} \leq \Vert\sigma_A\Vert_{L^p(G,\ell^{p}(\widehat{G}))}$$ in the first part of the proof of Theorem \ref{Michael:Criterion}.
\end{remark}

In the next result, we assume more regularity in the spatial variable to deduce a criterion for pseudo-differential operators to belong to the Schatten classes.
\begin{theorem}\label{Michael:criterion:2}Let $G$ be a compact Lie group of dimension $n.$ Let $A:C^\infty(G)\rightarrow C^\infty(G) $ be a continuous linear operator and assume that its matrix-valued symbol satisfies the regularity condition
\begin{equation}
    \Vert (1+\mathcal{L}_G)^{\frac{N}{2}}\sigma_A(x,\cdot)\Vert_{L^1(G,\mathscr{S}_p(\widehat{G}))}=\smallint\limits_G\Vert (1+\mathcal{L}_G)^{\frac{N}{2}}\sigma_A(x,\cdot)\Vert_{\mathscr{S}_p(\widehat{G})}\textnormal{d}x<\infty
\end{equation}where $N>n.$ Then $A\in \mathscr{S}_{p}(L^2(G))$ for all $1\leq p<\infty.$
\end{theorem}
\begin{proof}Let us consider the matrix-valued symbol $\sigma_A(x,[\xi])=(\sigma_{A,ij}(x,[\xi]))_{i,j=1}^{d_\xi}.$ The Fourier inversion formula allows one to write
\begin{equation}
    \sigma_{A,ij}=\sum_{[\eta]\in \widehat{G}}d_\eta\eta(x)_{rs}
    \widehat{\sigma}_{A,\,ij,sr}(\eta,\xi),
\end{equation} where $\widehat{\sigma}_{A,\,ij,sr}([\eta],[\xi])$ denotes the $(s,r)$-Fourier coefficient of the function $\sigma_{A,ij}(\cdot,[\xi])$ at $\eta\in [\eta]\in \widehat{G}.$ For any $f\in   C^{\infty}(G),$
the quantisation formula gives the identity
\begin{align*}
    Af(x) &=\sum_{[\xi]\in \widehat{G}}d_\xi\textnormal{Tr}[\xi(x)\sigma_A(x,[\xi])\widehat{f}(\xi)]=\sum_{[\xi]\in \widehat{G}}\sum_{i,j,\ell=1}^{d_\xi}d_\xi\xi_{ij}(x)\sigma_{A,\,ji }(x,[\xi])\widehat{f}(\xi)_{i\ell}\\
    &=\sum_{[\eta]\in \widehat{G}}\sum_{r,s=1}^{d_\eta}\sum_{[\xi]\in \widehat{G}}\sum_{i,j,\ell=1}^{d_\xi}d_\xi \xi_{ij}(x)\eta_{rs}(x)\widehat{\sigma}_{A,\,ji,\,sr }([\eta],[\xi])\widehat{f}(\xi)_{i\ell}\\
    &=\sum_{[\eta]\in \widehat{G}}\sum_{r,s=1}^{d_\eta}\sum_{[\xi]\in \widehat{G}}\eta_{rs}(x)d_\xi\textnormal{Tr}[ \xi(x)\widehat{\sigma}_{A,\,sr }([\eta],[\xi])\widehat{f}(\xi)],
\end{align*}where $$ \widehat{\sigma}_{A,\,sr }([\eta],[\xi]):=(\widehat{\sigma}_{A,\,ji,\,sr }([\eta],[\xi]))_{j,i=1}^{d_\xi}. $$ Let us define the  operator $\textnormal{Op}(\widehat{\sigma}_{A,\,sr }([\eta],\cdot))$ corresponding to the matrix $\widehat{\sigma}_{A,\,sr }([\eta],\cdot)$ . Let $M_{\eta_{rs}}$ be the multiplication operator associated to the function $\eta_{rs}.$ Note that $M_{\eta_{rs}}$ is a bounded operator on $L^2(G)$ with operator norm
\begin{equation}
    \Vert M_{\eta_{rs}}\Vert_{\mathscr{B}(L^2(G))}=\Vert\eta_{rs}\Vert_{L^\infty(G)}.
\end{equation}
Then, we have that
\begin{equation}
 \forall f\in C^{\infty}(G),\,\,\forall x\in G,\,
 Af(x) =\sum_{[\eta]\in \widehat{G}}\sum_{r,s=1}^{d_\eta}M_{\eta_{rs}}[\textnormal{Op}(\widehat{\sigma}_{A,\,sr }([\eta],\cdot ))f](x).
\end{equation}
Note that for all $N\in \mathbb{N},$ the symbol $$(x,[\xi])\mapsto \sigma_{N,rs}(x,[\xi]):=(1+\mathcal{L}_G)^{\frac{N}{2}}\sigma_{A,\,rs}(x,[\xi])$$  satisfies the Fourier transform identity
\begin{equation}
 \widehat{\sigma}_{A,\,sr }([\eta],[\xi]) =  \langle\eta\rangle^{-N}\widehat{\sigma}_{N,rs}([\eta],[\xi]).
\end{equation}This identity allows us to write
\begin{align*}
    \Vert A \Vert_{\mathscr{S}_p(L^2(G))} &\leq \sum_{[\eta]\in \widehat{G}}\sum_{r,s=1}^{d_\eta}\Vert M_{\eta_{rs}}\textnormal{Op}(\widehat{\sigma}_{A,\,sr }([\eta],\cdot))\Vert_{\mathscr{S}_p(L^2(G))}\\
    &=\sum_{[\eta]\in \widehat{G}}\sum_{r,s=1}^{d_\eta}\langle \eta\rangle^{-N}\Vert M_{\eta_{rs}}\textnormal{Op}(\widehat{\sigma}_{N,\,sr }([\eta],\cdot))\Vert_{\mathscr{S}_p(L^2(G))}\\
    &=\sum_{[\eta]\in \widehat{G}}\sum_{r,s=1}^{d_\eta}\langle \eta\rangle^{-N}\Vert M_{\eta_{rs}}\Vert_{\mathscr{B}(L^2)}\Vert\textnormal{Op}(\widehat{\sigma}_{N,\,sr }([\eta],\cdot ))\Vert_{\mathscr{S}_p(L^2(G))}\\
    &=\sum_{[\eta]\in \widehat{G}}\sum_{r,s=1}^{d_\eta}\langle \eta\rangle^{-N}\Vert \eta_{rs}\Vert_{L^\infty(G)}\Vert\textnormal{Op}(\widehat{\sigma}_{N,\,sr }([\eta],\cdot ))\Vert_{\mathscr{S}_p(L^2(G))}.
\end{align*}
Since
\begin{equation}\label{estimate:eta:sr}
 \forall x\in G,\,  |\eta_{rs}(x)|\leq \left(\sum_{r',s'=1}^{d_{\eta}}|\eta_{r's'}(x)|^2\right)^{\frac{1}{2}}=\sqrt{\textnormal{Tr}[\eta(x)\eta(x)^{*}]}=\sqrt{d}_\eta,
\end{equation}
we also have
\begin{equation}
   \forall x\in G,\,\sum_{r,s=1}^{d_\eta} |\eta_{rs}(x)|\leq \left(\sum_{r',s'=1}^{d_{\eta}}|\eta_{r's'}(x)|^2\right)^{\frac{1}{2}}d_\eta=\sqrt{\textnormal{Tr}[\eta(x)\eta(x)^{*}]}d_\eta=d_\eta\sqrt{d}_\eta,
\end{equation} where the last implies
\begin{equation}
  \,\sum_{r,s=1}^{d_\eta} \|\eta_{rs}\|_{L^\infty}\leq d_\eta\sqrt{d}_\eta.
\end{equation}
On the other hand, using the triangle inequality for the $\mathscr{S}_p$-norm we have that
\begin{align*}
  \Vert\textnormal{Op}(\widehat{\sigma}_{N,\,sr }([\eta],\cdot ))\Vert^{p}_{\mathscr{S}_p(L^2(G))}
  &\lesssim\sum_{[\xi]\in \widehat{G}}d_\xi  \Vert \widehat{\sigma}_{N,\,sr }([\eta],[\xi])\Vert_{\mathscr{S}_p(\widehat{G})}^{p}\\
  &=\sum_{[\xi]\in \widehat{G}}d_\xi  \Vert\smallint\limits_{G} {\sigma}_{N}(x,[\xi])\eta_{rs}(x)^{*}\textnormal{d}x\Vert_{\mathscr{S}_p(\widehat{G})}^{p}\\
  &\leq  \sum_{[\xi]\in \widehat{G}}d_\xi \left(\smallint\limits_{G}\Vert\eta_{rs}^{*}\Vert_{L^{\infty}(G)}  \Vert {\sigma}_{N}(x,[\xi])\Vert_{\mathscr{S}_p(\widehat{G})}\textnormal{d}x\right)^{p}\\
  &\leq  \sum_{[\xi]\in \widehat{G}}d_\xi\sqrt{d}_{\eta}^{p}\left( \smallint\limits_{G}  \Vert {\sigma}_{N}(x,[\xi])\Vert_{\mathscr{S}_p(\widehat{G})}\textnormal{d}x\right)^p,
\end{align*}
where, in the last line, we have used \eqref{estimate:eta:sr} to estimate
\begin{equation*}
  \Vert\eta_{rs}^{*}\Vert_{L^{\infty}(G)}^p=\Vert\overline{\eta}_{rs}\Vert_{L^{\infty}(G)}^p  =\Vert\eta_{rs}\Vert_{L^{\infty}(G)}^p  \leq \sqrt{d}_\eta^p.
\end{equation*}
The Minkowski integral inequality for $p\geq 1$ implies that
\begin{align*}
\Vert\textnormal{Op}(\widehat{\sigma}_{N,\,sr }([\eta],\cdot ))\Vert_{\mathscr{S}_p} &\lesssim\left( \sum_{[\xi]\in \widehat{G}}d_\xi\sqrt{d}_{\eta}^{p}\left( \smallint\limits_{G}  \Vert {\sigma}_{N }(x,[\xi])\Vert_{\mathscr{S}_p}\textnormal{d}x\right)^p  \right)^{\frac{1}{p}} \\
&\leq \sqrt{d}_\eta \smallint\limits_G\left(\sum_{[\xi]\in \widehat{G}}d_\xi \Vert {\sigma}_{N }(x,[\xi])\Vert_{\mathscr{S}_p}^p \right)^{\frac{1}{p}}\textnormal{d}x\\
&= \sqrt{d}_\eta \smallint\limits_G\Vert {\sigma}_{N }(x,\cdot)\Vert_{\mathscr{S}_p(\widehat{G})} \textnormal{d}x\\
&= \sqrt{d}_\eta \Vert {\sigma}_{N }(x,\cdot)\Vert_{L^1(G,\mathscr{S}_p(\widehat{G}))}.
\end{align*}
All the analysis above allows the following estimate  for the $p$-Schatten norm of $A$ as follows:
\begin{align*}
  \Vert Af \Vert_{\mathscr{S}_p} &\lesssim \sum_{[\eta]\in \widehat{G}}\sum_{r,s=1}^{d_\eta}\langle \eta\rangle^{-N}\Vert \eta_{rs}\Vert_{L^\infty(G)}\Vert\textnormal{Op}(\widehat{\sigma}_{N,\,sr }([\eta],\cdot ))\Vert_{\mathscr{S}_p}  \\
  &\lesssim \sum_{[\eta]\in \widehat{G}}\sum_{r,s=1}^{d_\eta}\langle \eta\rangle^{-N}\Vert \eta_{rs}\Vert_{L^\infty(G)}\sqrt{d}_\eta \Vert {\sigma}_{N }(x,\cdot)\Vert_{L^1(G,\mathscr{S}_p(\widehat{G}))}  \\
  &= \Vert {\sigma}_{N }(x,\cdot)\Vert_{L^1(G,\mathscr{S}_p(\widehat{G}))}  \sum_{[\eta]\in \widehat{G}}\langle \eta\rangle^{-N}\sqrt{d}_\eta \sum_{r,s=1}^{d_\eta}\Vert \eta_{rs}\Vert_{L^\infty(G)}  \\
   &\leq  \Vert {\sigma}_{N }(x,\cdot)\Vert_{L^1(G,\mathscr{S}_p(\widehat{G}))}  \sum_{[\eta]\in \widehat{G}}\langle \eta\rangle^{-N}\sqrt{d}_\eta d_\eta \sqrt{d}_\eta.
\end{align*} Thus, for $N>n$ where $n=\dim(G)$ we have that 
\begin{equation}
  \Vert Af \Vert_{\mathscr{S}_p} \lesssim    \Vert {\sigma}_{N }(x,\cdot)\Vert_{L^1(G,\mathscr{S}_p(\widehat{G}))} \sum_{[\eta]\in \widehat{G}}d_\eta^2\langle \eta\rangle^{-N}<\infty,
\end{equation} where we have use that (see Lemma 3.8 of \cite{DelRuzTrace11})
$$ \sum_{[\eta]\in \widehat{G}}d_\eta^2\langle \eta\rangle^{-s}<\infty \Longleftrightarrow s>n. $$
The proof of Theorem \ref{Michael:criterion:2} is complete.
\end{proof}
The following is a consequence of Theorem \ref{Michael:criterion:2}.
\begin{corollary}Let $G$ be a compact Lie group of dimension $n.$ Let $A:C^\infty(G)\rightarrow C^\infty(G) $ be a continuous linear operator and assume that its matrix-valued symbol satisfies the regularity condition
\begin{equation}
    \Vert (1+\mathcal{L}_G)^{\frac{N}{2}}\sigma_A(x,\cdot)\Vert_{L^p(G,\mathscr{S}_p(\widehat{G}))}=\left(\smallint\limits_G\Vert (1+\mathcal{L}_G)^{\frac{N}{2}}\sigma_A(x,\cdot)\Vert_{\mathscr{S}_p(\widehat{G})}^p\textnormal{d}x\right)^{\frac{1}{p}}<\infty
\end{equation}where $N>n.$ Then $A\in \mathscr{S}_{p}(L^2(G))$ provided that $1\leq p<\infty.$
\end{corollary}
\begin{proof}Using the H\"older inequality and that the Haar measure of $G$ is normalised we have that
\begin{align}
   \smallint\limits_G\Vert (1+\mathcal{L}_G)^{\frac{N}{2}}\sigma_A(x,\cdot)\Vert_{\mathscr{S}_p(\widehat{G})}\textnormal{d}x\leq \left( \smallint\limits_G\Vert (1+\mathcal{L}_G)^{\frac{N}{2}}\sigma_A(x,\cdot)\Vert_{\mathscr{S}_p(\widehat{G})}^p\textnormal{d}x  \right)^{\frac{1}{p}}<\infty,
\end{align}for all $N>n.$ Then, in view of Theorem \ref{Michael:criterion:2} follows the membership of $A$ in the Schatten class $\mathscr{S}_p(L^2(G)).$    
\end{proof}

\subsection{Schatten properties of elliptic operators. The $(\rho,\delta)$-case}\label{rhodeltasection}

Now we characterise the membership of elliptic pseudo-differential operators to the Schatten classes on $L^2(G).$
\begin{theorem}\label{Schatten:properties:}  Let $m\in \mathbb{R},$ $r>0, $ and let $0\leq \delta<\rho\leq 1.$ Consider an elliptic pseudo-differential operator $A\in \Psi^m_{\rho,\delta}(G\times \widehat{G})$. The following conditions are equivalent:
\begin{itemize}
    \item[(1)] $m\in (-\infty,0)$ and  $A$ belongs to the Schatten class of order $r>0:$ $A\in \mathscr{S}_r(L^2(G));$
    \item[(2)] The Bessel potential of order $m$ belongs to the Schatten class of order $r>0:$ $B_{m}:=(1+\mathcal{L}_G)^{\frac{m}{2}}\in \mathscr{S}_r(L^2(G)). $
    \item[(3)] The following summability condition holds
    \begin{equation}
        \sum_{[\xi]\in \widehat{G}}d_\xi\smallint\limits_G\|\sigma_{|A|^{\frac{r}{2}}}(x,[\xi])\|_{\textnormal{HS} }^2\textnormal{d}x<\infty;
    \end{equation}
    \item[(4)] The following summability condition holds
    \begin{equation}
        \smallint\limits_G\sum_{[\xi]\in \widehat{G}}d_\xi\|\sigma_{A}(x,[\xi])\|_{\mathscr{S}_r }^r\textnormal{d}x<\infty;
    \end{equation}
    \item[(5)] $m<-n/r.$
\end{itemize}
\end{theorem}
\begin{proof}We start with the first equivalence $(1)\Longleftrightarrow(2):$

    \begin{itemize}
    \item $(1)\Longrightarrow (2).$ Assume that $A\in \mathscr{S}_r(L^2(G)).$  Let  $B_{s}:=(1+\mathcal{L}_G)^{\frac{s}{2}}\in \mathscr{S}_r(L^2(G))$ be the Bessel potential of order $s\in \mathbb{R}.$ The global functional calculus for elliptic operators  on compact Lie groups (see \cite{RuzhanskyWirth2014}) implies that
    \begin{equation}
        (1+|A|^{\frac{1}{|m|}})^{ \pm {m}}\in \Psi^{ \pm m}_{\rho,\delta}(G\times \widehat{G}).
    \end{equation} In view of the composition properties of the pseudo-differential calculus we have that $$  B_{m}(1+|A|^{\frac{1}{|m|}})^{ -m }\in \Psi^{0}_{\rho,\delta}(G\times \widehat{G}).$$ The Calder\'on-Vaillancourt theorem (see Theorem 5.2 in \cite{RuzhanskyWirth2014}) implies that $B_{m}(1+|A|^{\frac{1}{|m|}})^{ -m }$ is a bounded operator on $L^2(G).$ Since $\mathscr{S}_{r}(L^2(G))$ is an ideal of operators on the algebra $\mathscr{B}(L^2(G)),$ we have that
    \begin{equation}
        B_{m}= B_{m}(1+|A|^{\frac{1}{|m|}})^{ -m }(1+|A|^{\frac{1}{|m|}})^{ m },
    \end{equation}belongs to the Schatten class $\mathscr{S}_r(L^2(G))$ provided that $$ (1+|A|^{\frac{1}{|m|}})^{ m }\in \mathscr{S}_r(L^2(G)).$$  Now, we are going to prove this fact. Since $m\leq 0,$ if $s_{j}(A)$ is the $j$-singular value of $A,$ then the $j$-singular value  $s_{j}((1+|A|^{\frac{1}{|m|}})^{ m })$ of $(1+|A|^{\frac{1}{|m|}})^{ m }$ satisfies the inequality
    \begin{equation}
       s_{j}((1+|A|^{\frac{1}{|m|}})^{ m })\leq s_{j}(A),
    \end{equation}and then the membership of $(1+|A|^{\frac{1}{|m|}})^{ m }$ in $\mathscr{S}_r(L^2(G))$ follows from (1).
    \end{itemize}
    \begin{itemize}
        \item $(2)\Longleftrightarrow(5).$ Observe that 
\begin{align*}
    B_{m}\in\mathscr{S}_r(L^2(G))\Longleftrightarrow B_{mr/2}\in\mathscr{S}_2(L^2(G)). 
\end{align*}On the other hand, 
\begin{align*}
    \Vert B_{mr/2}\Vert_{\mathscr{S}_2(L^2(G))}^2=\sum_{[\xi]\in \widehat{G}}d_\xi^2\langle \xi \rangle^{mr}<\infty,
\end{align*}if and only if $mr<-n,$ (see Lemma 3.8 of \cite{DelRuzTrace11}) or equivalently, if $m<-n/r$ proving the statement.
    \end{itemize}
    \begin{itemize}
    \item $(2)\Longrightarrow (1).$ Let us give a similar proof to the one given for the reverse statement. Because of $(2)$ we have that $m<-n/r<0.$ On the other hand,    by the pseudo-differential calculus, we have that
    \begin{equation}
            AB_{-m}\in \Psi^{0}_{\rho,\delta}(G\times \widehat{G}).
        \end{equation}  Moreover, the Calder\'on-Vaillancourt theorem (see Theorem 5.2 in \cite{RuzhanskyWirth2014})  gives the boundedness of $AB_{-m}$ on $L^2(G).$ If $B_{m}\in\mathscr{S}_r(L^2(G)),$ then using that $\mathscr{S}_{r}(L^2(G))$ is an ideal in the algebra of the bounded operators on $L^2(G),$ we have that
        \begin{equation*}
            A=AB_{-m}B_{m}\in\mathscr{S}_r(L^2(G)).
        \end{equation*}
\item $(1)\Longleftrightarrow(3).$ Note that
\begin{equation}
    A\in \mathscr{S}_{r}(L^2(G))\Longleftrightarrow |A|^{\frac{r}{2}}\in\mathscr{S}_2(L^2(G)).
\end{equation}Let $R_{|A|^{\frac{r}{2}}}(x,y) $ and $K_{|A|^{\frac{r}{2}}}(x,y) $ be the right-convolution kernel and the Schwartz kernel  of the operator $|A|^{\frac{r}{2}},$ respectively. We have the identity $$ \forall x,y\in G,\,\, R_{A}(x,xy^{-1})=K_A(x,y).$$   By the Plancherel theorem we have the equivalences
\begin{align*}
    |A|^{\frac{r}{2}}\in\mathscr{S}_2(L^2(G))  &\Longleftrightarrow K_{|A|^{\frac{r}{2}}}(x,y)\in L^2(G\times G)\\
    &\Longleftrightarrow \smallint\limits_{G}\smallint\limits_{G}|R_{A}(x,y)|^2\textnormal{d}y\textnormal{d}x<\infty\\
    &\Longleftrightarrow \smallint\limits_{G}\sum_{[\xi]\in \widehat{G}} d_\xi\Vert\sigma_{|A|^{\frac{r}{2}}} (x,[\xi])\|^2_{\textnormal{HS}}\textnormal{d}x<\infty.
\end{align*}

\item $(3)\Longleftrightarrow(4).$ In order to prove the equivalence of these summability conditions we are going to apply the global functional calculus (see \cite{RuzhanskyWirth2014}). We start with the identity
\begin{equation}
    \sigma_{|A|^{\frac{r}{2}}}(x,[\xi])=|\sigma_A(x,[\xi])|^{\frac{r}{2}}+r_{A}(x,[\xi]),
\end{equation}where the lower term $r_{A}\in S^{\frac{mr}{2}-(\rho-\delta)}(G\times \widehat{G})$ because of the asymptotic expansions. First, note that
\begin{align*}
    \Vert\sigma_{|A|^{\frac{r}{2}}}(x,[\xi])\Vert_{\textnormal{HS}} &\leq  \Vert  |\sigma_A(x,[\xi])|^{\frac{r}{2}} \Vert_{\textnormal{HS}}+ \Vert r_{A}(x,[\xi])\Vert_{\textnormal{HS}}\\
    &=\Vert  \sigma_A(x,[\xi]) \Vert_{\mathscr{S}_r}^{\frac{r}{2}}+ \Vert r_{A}(x,[\xi])\Vert_{\textnormal{HS}}.
\end{align*}Similarly, since $|\sigma_A(x,[\xi])|^{\frac{r}{2}}=\sigma_{|A|^{\frac{r}{2}}}(x,[\xi])-r_{A}(x,[\xi]),$ we have that 
        \begin{align*}
           \Vert  \sigma_A(x,[\xi]) \Vert_{\mathscr{S}_r}^{\frac{r}{2}}= \Vert  |\sigma_A(x,[\xi])|^{\frac{r}{2}} \Vert_{\textnormal{HS}}\leq \Vert\sigma_{|A|^{\frac{r}{2}}}(x,[\xi])\Vert_{\textnormal{HS}} + \Vert r_{A}(x,[\xi])\Vert_{\textnormal{HS}}.
        \end{align*}
        To estimate the remainder $\Vert r_{A}(x,[\xi])\Vert_{\textnormal{HS}}$ let us use that $r_A\in S^{ \frac{mr}{2}-(\rho-\delta)}(G\times \widehat{G}). $ Then,
   \begin{equation}
      \Vert r_{A}(x,[\xi]) \Vert_{\textnormal{HS}}\leq \Vert r_{A}(x,[\xi])\langle\xi\rangle^{-mr/2+(\rho-\delta)}\Vert_{\textnormal{op}}\Vert \langle\xi\rangle^{mr/2-(\rho-\delta)}I_{d_\xi} \Vert_{\textnormal{HS}} \lesssim d_{\xi}^{\frac{1}{2}}\langle\xi\rangle^{\frac{mr}{2}-(\rho-\delta)}.
   \end{equation} To prove that $(3)\Longrightarrow(4),$ observe that the condition $(3)$ is equivalent to the fact that $A\in \mathscr{S}_{r}(L^2(G)),$ from which we deduce that $B_{m}\in \mathscr{S}_{r}(L^2(G)).$ In terms of the order $m$ we have that $m<-n/r.$ To prove that $(4)$ holds, note that
   \begin{align*}
      \smallint\limits_G\sum_{[\xi]\in \widehat{G}}d_\xi \Vert  \sigma_A(x,[\xi]) \Vert_{\mathscr{S}_r}^{r}\textnormal{d}x &\leq  \smallint\limits_G\sum_{[\xi]\in \widehat{G}}d_\xi \max\{ 2\Vert\sigma_{|A|^{\frac{r}{2}}}(x,[\xi])\Vert_{\textnormal{HS}}, \Vert r_{A}(x,[\xi])\Vert_{\textnormal{HS}}\}^2\\
     &= 4\max\left\{\smallint\limits_G\sum_{[\xi]\in \widehat{G}}d_\xi  \Vert\sigma_{|A|^{\frac{r}{2}}}(x,[\xi])\Vert_{\textnormal{HS}}^2\textnormal{d}x,\smallint\limits_G\sum_{[\xi]\in \widehat{G}}d_\xi \Vert r_{A}(x,[\xi])\Vert_{\textnormal{HS}}^2\right\}\\
      &\lesssim  4\max\left\{\smallint\limits_G\sum_{[\xi]\in \widehat{G}}d_\xi  \Vert\sigma_{|A|^{\frac{r}{2}}}(x,[\xi])\Vert_{\textnormal{HS}}^2\textnormal{d}x,\sum_{[\xi]\in \widehat{G}}d_\xi^2 \langle \xi\rangle^{mr-2(\rho-\delta)}\right\}.
   \end{align*}Since $\rho>\delta,$ we can compare $\langle \xi\rangle^{mr-2(\rho-\delta)}\leq \langle \xi\rangle^{mr},$ and then
   \begin{equation}\label{Aux:convergence}
       \sum_{[\xi]\in \widehat{G}}d_\xi^2 \langle \xi\rangle^{mr-2(\rho-\delta)}\leq \sum_{[\xi]\in \widehat{G}}d_\xi^2 \langle \xi\rangle^{mr}<\infty
   \end{equation}because $mr<-n.$ Then, $(3)$ together with the convergence of the series in \eqref{Aux:convergence} implies $(4)$. Now, to finish the proof, let us assume $(4)$ and from it let us deduce $(3).$ To do so we are going to use the inequality
   \begin{equation}
     \Vert\sigma_{|A|^{\frac{r}{2}}}(x,[\xi])\Vert_{\textnormal{HS}}    \leq \Vert  \sigma_A(x,[\xi]) \Vert_{\mathscr{S}_r}^{\frac{r}{2}}+ \Vert r_{A}(x,[\xi])\Vert_{\textnormal{HS}}.
   \end{equation} Then, we argue as follows:
   \begin{align*}
      \smallint\limits_G\sum_{[\xi]\in \widehat{G}}d_\xi \Vert\sigma_{|A|^{\frac{r}{2}}}(x,[\xi])\Vert_{\textnormal{HS}}^2 \textnormal{d}x &\leq  \smallint\limits_G\sum_{[\xi]\in \widehat{G}}d_\xi (2\max\{ \Vert  \sigma_A(x,[\xi]) \Vert_{\mathscr{S}_r}^{\frac{r}{2}}, \Vert r_{A}(x,[\xi])\Vert_{\textnormal{HS}}\})^2\\
     &= 4\max\left\{\smallint\limits_G\sum_{[\xi]\in \widehat{G}}d_\xi  \Vert  \sigma_A(x,[\xi]) \Vert_{\mathscr{S}_r}^{r}\textnormal{d}x,\smallint\limits_G\sum_{[\xi]\in \widehat{G}}d_\xi \Vert r_{A}(x,[\xi])\Vert_{\textnormal{HS}}^2\right\}\\
      &\lesssim  4\max\left\{\smallint\limits_G\sum_{[\xi]\in \widehat{G}}d_\xi  \Vert\sigma_{|A|^{\frac{r}{2}}}(x,[\xi])\Vert_{\textnormal{HS}}^2\textnormal{d}x,\sum_{[\xi]\in \widehat{G}}d_\xi^2 \langle \xi\rangle^{mr-2(\rho-\delta)}\right\}
   \end{align*}from which we deduce the convergence of the series in (3).
     \end{itemize} In view of the analysis above the proof of Theorem \ref{Schatten:properties:} is complete.
\end{proof}

\subsection{Schatten properties of non-elliptic operators. Classical symbols}\label{ClassicalSection}

We start this subsection, by showing that the order of an operator can be used as a sufficient condition  for  deducing its Schatten-von-Neumann properties, even if the operator is non-elliptic. 

The next corollary is a consequence of Theorem \ref{Schatten:properties:}.

\begin{corollary}\label{Schatten:properties:cor}  Let $m<-n/r,$  $r>0, $ and let $0\leq \delta<\rho\leq 1.$ Consider a pseudo-differential operator $A\in \Psi^m_{\rho,\delta}(G\times \widehat{G})$. Then, $A\in \mathscr{S}_{r}(L^2(G)).$
\end{corollary}
\begin{proof}
    Let us give an algebraic argument. Since $m<-n/r,$ then the Bessel potential satisfies $B_{m}=(1+\mathcal{L}_G)^{\frac{m}{2}}\in \mathscr{S}_{r}(L^2(G)).$ Using again that the class $\mathscr{S}_{r}(L^2(G))$ is an ideal on the algebra of all bounded operators on $L^2(G),$ and that $AB_{-m}$ is bounded on $L^2(G)$ (in view of the Calder\'on-Vaillancourt theorem), we have that $A=AB_{-m} B_{m}$ belongs to the Schatten class $\mathscr{S}_{r}(L^2(G)).$
\end{proof}
Now, using the local Weyl formula (see  \cite[Theorem 1.1]{CardonaDelgadoRuzhanksyLOcalWeyl}) for elliptic operators we have the following improvement of the equivalence $\textnormal{(1)}\Longleftrightarrow \textnormal{(5)} $ in Theorem \ref{Schatten:properties:} for trace class operators in the Kohn-Nirenberg algebra. Here,  let us consider  the norm $|| \cdot ||_{g}$ on the Lie algebra $\mathfrak{g}$ induced by the Killing form $B(X,Y)=\textnormal{Tr}[\textnormal{ad}(X)\textnormal{ad}(Y)],$ $X,Y\in\mathfrak{g},  $ defined by  $|| X ||_{g}:=\sqrt{-B(X,X)}.$
\begin{proposition}\label{Invariance:trace} Let $A\in \Psi^{m}_{1,0}(G)$ be a classical  pseudo-differential operator of order $m\leq 0$. Assume that the average of the principal symbol $\sigma_{loc,A}\in C^\infty(T^*G)$ of $A$ on the co-sphere is non-zero, that is
\begin{equation}\label{Average:symbol}
    \textnormal{Av}[\sigma_{loc,A}]:=\smallint\limits_{ T^{*}\mathbb{S}(G)}\sigma_{loc,A}(x,\eta)d\mu_{L}(x,\eta)\neq 0,
\end{equation}
with $\mu_{L}$ denoting the Liouville measure on the spherical vector bundle $T^*\mathbb{S}(G)=\{(x,\eta)\in T^*G:||\eta||_g=1\}.$    Then, $A\in \mathscr{S}_1(L^2(G))$ if and only if $m<-n.$    
\end{proposition}
\begin{proof} From Corollary \ref{Schatten:properties:cor} it follows that for $m<-n,$ we have $A\in \mathscr{S}_1(L^2(G)).$ Now, let us prove the converse statement, that is if $A\in \mathscr{S}_1(L^2(G))$ then $m<-n.$ Let us prove this by assuming that $m\geq -n,$   and let us get a contradiction.

Note that if $A\in \mathscr{S}_1(L^2(G))$ then for any orthonormal basis $(\phi_k)_{k}$ of $L^2(G),$ the series $\sum_{k}(A\phi_k,\phi_k)_{L^2(G)}$ is absolutely convergent  and this sum is independent of the choice of the orthonormal basis $(\phi_k).$ Then, the trace of $A$ is given by 
\begin{equation}
    \textnormal{Tr}(A):=\sum_{k}(A\phi_k,\phi_k)_{L^2(G)}.
\end{equation}
For our purposes, we will consider the orthonormal basis
\begin{equation}\label{Peter:Weyl}
 \{d_{\xi}^{\frac{1}{2}}\xi_{i,j}:{[\xi]\in \widehat{G}:1\leq i,j\leq d_\xi}\},   
\end{equation}
 of $L^2(G)$ provided by the Peter-Weyl theorem. Consider the spectrum of the positive Laplacian $$ \textnormal{Spect}(\mathcal{L}_G)=\{|\xi|:=\lambda_{[\xi]}:[\xi]\in \widehat{G}\}.$$  For any $\lambda>0,$ we will consider the partial sum
\begin{equation}
    S_{\lambda}:=  \sum_{|\xi|\leq \lambda}  \sum_{i,j=1}^{d_\xi}(\tilde{A}(d_\xi^{\frac{1}{2}}\xi_{ij}),d_\xi^{\frac{1}{2}}\xi_{ij}),
\end{equation} and the discussion above allows us to use the identity
\begin{equation}\label{TraceofA}
    \textnormal{Tr}(A)=\lim_{\lambda\rightarrow \infty}S_{\lambda}=\lim_{\lambda\rightarrow \infty}\sum_{|\xi|\leq \lambda}  \sum_{i,j=1}^{d_\xi}(\tilde{A}(d_\xi^{\frac{1}{2}}\xi_{ij}),d_\xi^{\frac{1}{2}}\xi_{ij}).
\end{equation}Next, we  will prove that if $m\geq -n,$ then the series in \eqref{TraceofA} is not absolutely convergent and then, that $A$ is not of trace class which would contradict our hypothesis.

To simplify the proof, let us consider the real part $\textnormal{Re}(A)$ and the imaginary part $\textnormal{Im}(A)$ of $A,$ respectively. We have the identities
 \begin{equation}
        \textnormal{Re}(A)=(A+A^{*})/2,\,\,\textnormal{Re}(A)=(A-A^{*})/2i.
    \end{equation}To simplify the notation we will write $A_{0}=\textnormal{Re}(A),$ and  $A_{1}=\textnormal{Im}(A).$ 
The following facts hold:
\begin{itemize}
    \item $A_{i}\in \Psi^m(G),$ $i=0,1,$ in view of the pseudo-differential calculus properties for sums operators.
    \item $A_{i},$ $i=0,1$, are self-adjoint operators.
\end{itemize}Note that any entry $(A_{i}(\xi_{ij}),\xi_{ij})\in \mathbb{R
}$ is a real number because of the self-adjointness of $A.$ Now, let us use a specific property of the Peter-Weyl basis \eqref{Peter:Weyl}. 
In \cite{CardonaDelgadoRuzhanksyLOcalWeyl}
the following Local Weyl-formula was obtained for any pseudo-differential operator $\tilde{A}\in \Psi^0(G):$
\begin{equation}\label{Preliminary:formula}
  \sum_{|\xi|\leq \lambda}  \sum_{i,j=1}^{d_\xi}(\tilde{A}(d_\xi^{\frac{1}{2}}\xi_{ij}),d_\xi^{\frac{1}{2}}\xi_{ij})= (2\pi)^{-n}C_{n,\tilde A}\lambda^n+O(\lambda^{n-1}),
\end{equation}for any $\lambda>0.$ Moreover, in \cite{CardonaDelgadoRuzhanksyLOcalWeyl} the left-hand side of this identity was simplified as follows
\begin{equation}\label{equality}
    \sum_{|\xi|\leq \lambda}  \sum_{i,j=1}^{d_\xi}d_\xi(\tilde A\xi_{ij},\xi_{ij})=  \sum_{|\xi|\leq \lambda}d_\xi\smallint\limits_{G}\textnormal{Tr}[\sigma_{\tilde A}(x,\xi)]\textnormal{d}x.
\end{equation}Let \begin{equation}
       \tilde A_{0}:= \textnormal{Re}(\tilde A)=(\tilde A+{\tilde A}^{*})/2,\,\,\tilde A_{1}:=\textnormal{Im}(\tilde A)=(\tilde A-{\tilde A}^{*})/2i.
    \end{equation}

Because these identities are valid for general operators of order $m>0,$ for $k=0,1$ we also have that
\begin{equation}\label{Preliminary:formula:2}
  \sum_{|\xi|\leq \lambda}  \sum_{i,j=1}^{d_\xi}(\tilde A_k(d_\xi^{\frac{1}{2}}\xi_{ij}),d_\xi^{\frac{1}{2}}\xi_{ij})= (2\pi)^{-n}C_{n,\tilde A_k}\lambda^n+O(\lambda^{n-1}),
\end{equation}for any $\lambda>0,$ where  the constant $C_{n,\tilde{A}_k} $ is given by
\begin{equation}\label{The:Constant}
  C_{n,\tilde{A}_k}=   \textnormal{Av}[\sigma_{loc,\tilde A_k}]:=\smallint\limits_{ T^{*}\mathbb{S}(G)}\sigma_{loc,\tilde A_k}(x,\eta)d\mu_{L}(x,\eta).
\end{equation}
As before we can write
\begin{equation}\label{eqref:symmetry}
    \sum_{|\xi|\leq \lambda}  \sum_{i,j=1}^{d_\xi}d_\xi(\tilde A_k\xi_{ij},\xi_{ij})=  \sum_{|\xi|\leq \lambda}d_\xi\smallint\limits_{G}\textnormal{Tr}[\sigma_{\tilde A_k}(x,\xi)]\textnormal{d}x,
\end{equation}where that the left hand side of \eqref{eqref:symmetry} is real valued. Note that in view of \eqref{Preliminary:formula:2} and \eqref{eqref:symmetry}  we have the identity
\begin{equation}\label{eqref:symmetry:22}
      \sum_{|\xi|\leq \lambda}d_\xi\smallint\limits_{G}\textnormal{Tr}[\sigma_{\tilde A_k}(x,\xi)]\textnormal{d}x= (2\pi)^{-n}C_{n,\tilde A_k}\lambda^n+O(\lambda^{n-1}).
\end{equation} Note that
\begin{align*}
    \sum_{\frac{\lambda}{2}<|\xi|\leq \lambda}d_\xi\smallint\limits_{G}\textnormal{Tr}[\sigma_{\tilde A_k}(x,\xi)]\textnormal{d}x&=\sum_{|\xi|\leq \lambda}d_\xi\smallint\limits_{G}\textnormal{Tr}[\sigma_{\tilde A_k}(x,\xi)]\textnormal{d}x-\sum_{|\xi|\leq\frac{\lambda}{2} }d_\xi\smallint\limits_{G}\textnormal{Tr}[\sigma_{\tilde A_k}(x,\xi)]\textnormal{d}x\\
    &=(2\pi)^{-n}C_{n,A_k}(1-\frac{1}{2^n})\lambda^n+O(\lambda^{n-1})\\
    &=(2\pi)^{-n}\tilde{C}_{n,\tilde A_k}\lambda^n+O(\lambda^{n-1}).
\end{align*}Let us apply the local Weyl formula above to the operator $\tilde A=A\mathcal{L}_G^{-\frac{m}{2}}\in \Psi^{0}(G).$  The matrix-valued symbol of $\tilde A$ and $\tilde A_k,$ $k=0,1,$ are given by $$\sigma_{\tilde A}(x,[\xi])=\sigma_{A}(x,[\xi])| \xi |^{-m},\,\,\sigma_{\tilde A_k}(x,[\xi])=\sigma_{A_k}(x,[\xi])|\xi|^{-m},\,\, (x,[\xi])\in G\times \widehat{G},$$ where 
$$|\xi|^{-m}:=\sqrt{\lambda_{[\xi]}}^{-m},\,\textnormal{ if }\xi\neq 1_{\widehat{G}}, \,\,|1_{\widehat{G}}|^{-m}:=0,$$ where $1_{\widehat{G}}$ is the trivial representation. 
Because of the self-adjointness of $\tilde A_k,$ the left-hand side (and the right-hand side) of \eqref{eqref:symmetry:22} is real-valued. 
Note that the linearity of the trace gives the identity $$ \textnormal{Tr}(A)= \textnormal{Tr}(A_0)+i\textnormal{Tr}(A_1).$$ So, let us  estimates the traces $\textnormal{Tr}(A_k),$ $k=0,1.$
Indeed,
\begin{align*}
    \textnormal{Tr}(A_k)&=\lim_{\lambda\rightarrow \infty} \sum_{i,j=1}^{d_\xi}d_\xi( A_k\xi_{ij},\xi_{ij})= \lim_{\lambda\rightarrow \infty}  \sum_{|\xi|\leq \lambda}d_\xi\smallint\limits_{G}\textnormal{Tr}[\sigma_{ A_k}(x,\xi)]\textnormal{d}x\\
    &=\sum_{k=0}^\infty  \sum_{[\xi]\in \widehat{G}:2^{k-1}<\langle \xi\rangle \leq 2^{k}}d_\xi\smallint\limits_{G}\textnormal{Tr}[\sigma_{ A_k}(x,\xi)]\textnormal{d}x\\
    &=\sum_{k=0}^\infty 2^{km}  \sum_{[\xi]\in \widehat{G}:2^{k-1}<\langle \xi\rangle \leq 2^{k}}d_\xi  2^{-km} \smallint\limits_{G}\textnormal{Tr}[\sigma_{ A_k}(x,\xi)]\textnormal{d}x\\
    &\asymp \sum_{k=0}^\infty 2^{km}  \sum_{[\xi]\in \widehat{G}:2^{k-1}<\langle \xi\rangle \leq 2^{k}}d_\xi \langle \xi\rangle^{-m} \smallint\limits_{G}\textnormal{Tr}[\sigma_{ A_k}(x,\xi)]\textnormal{d}x\\
     &=\sum_{k=0}^\infty 2^{km}  \sum_{[\xi]\in \widehat{G}:2^{k-1}<\langle \xi\rangle \leq 2^{k}}d_\xi  \smallint\limits_{G}\textnormal{Tr}[\sigma_{ A_k}(x,\xi)\langle \xi\rangle^{-m}]\textnormal{d}x\\
     &\asymp \sum_{k=0}^\infty 2^{km}  \sum_{[\xi]\in \widehat{G}:2^{k-1}<\langle \xi\rangle \leq 2^{k}}d_\xi  \smallint\limits_{G}\textnormal{Tr}[\sigma_{ A_k}(x,\xi)|\xi|^{-m}]\textnormal{d}x\\
     &\asymp \sum_{k=0}^\infty 2^{km}  \sum_{[\xi]\in \widehat{G}:2^{k-1}<| \xi| \leq 2^{k}}d_\xi  \smallint\limits_{G}\textnormal{Tr}[\sigma_{ \tilde A_k}(x,\xi)]\textnormal{d}x.
\end{align*}Now, using the local Weyl formula we get 
\begin{align*}
  \textnormal{Tr}(A_k)&\asymp \sum_{k=0}^\infty 2^{km}  \sum_{[\xi]\in \widehat{G}:2^{k-1}<| \xi| \leq 2^{k}}d_\xi  \smallint\limits_{G}\textnormal{Tr}[\sigma_{ \tilde A_k}(x,\xi)]\textnormal{d}x\\
  &\asymp  \sum_{k=0}^\infty 2^{km} \left( (2\pi)^{-n}\tilde{C}_{n,\tilde A_k}2^{kn}+O(2^{k(n-1)}\right)\\
  &\asymp  \sum_{k=0}^\infty \left( \tilde{C}_{n,\tilde A_k}2^{k(n+m)}+O(2^{k(n+m-1)}\right).
\end{align*}Now, it is clear that if $m\geq -n,$then  the geometric series $$  \sum_{k=0}^\infty \left( \tilde{C}_{n,\tilde A_k}2^{k(n+m)}+O(2^{k(n+m-1)}\right)$$  diverges for $k=0$ or for $k=1,$ which also implies that the trace of $T$ diverges. To see this observe that the constant in \eqref{The:Constant} cannot be equal to zero simultaneously for $k=0,1.$ Indeed, from \eqref{Average:symbol} we have that
\begin{equation}
    \textnormal{Av}[\sigma_{loc,A}]=\tilde{C}_{n,\tilde A_0}+i\tilde{C}_{n,\tilde A_1}\neq 0.
\end{equation}
We illustrate this geometric fact in Figure  \ref{Average} below.
\begin{figure}[h]
\includegraphics[width=8.5cm]{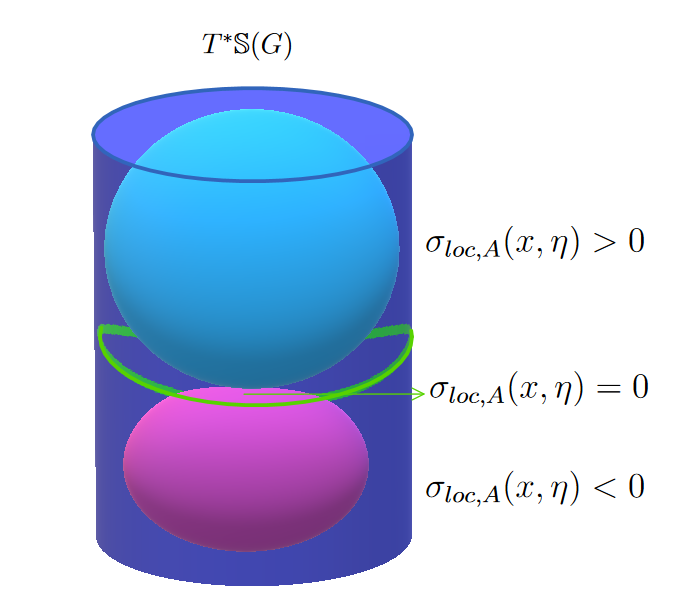}\\
\caption{The average condition on the co-sphere in the case of a real-valued principal symbol $\sigma_A.$ In this case the positive part of the symbol dominates the region where the symbol is negative. The green curve represents the level curve at height zero (that occurs along the zero section).}
 \label{Average}
\centering
\end{figure} 
Note that we have used that on the co-sphere $T^{*}\mathbb{S}(G)$ one has the identity
\begin{equation}
    \forall(x,\eta)\in T^{*}\mathbb{S}(G),\,\sigma_{\tilde A}(x,\eta)=\sigma_{ A}(x,\eta)\| \eta\|^{-m}=\sigma_{ A}(x,\eta),
\end{equation}because the fact that $(x,\eta)\in T^*\mathbb{S}(G)$ implies that $x\in M$ and that the norm $\|\eta\|$ of $\eta$ on the corresponding fiber is equal to one.

Thus we have proved that $T$ is not in the ideal $\mathscr{S}_1(L^2(G))$ which contradicts our initial hypothesis. The proof of Proposition \ref{Invariance:trace} is complete.
\end{proof}

As a consequence of the previous corollary we have the following characterisation of the trace class pseudo-differential operators.

\begin{theorem}\label{Invariance:trace2} Let $A\in \Psi^{m}_{1,0}(G)$ be a classical  pseudo-differential operator of order $m\leq 0$. 
Then, $A\in \mathscr{S}_1(L^2(G))$ if and only if $m<-n.$    
\end{theorem}
\begin{proof} Note that if the average of the principal symbol $\sigma_{loc,A}\in C^\infty(T^*G)$ of $A$ on the co-sphere is non-zero, that is
\begin{equation}\label{Average:symbol:traceclass:trace}
    \textnormal{Av}[\sigma_{loc,A}]:=\smallint\limits_{ T^{*}\mathbb{S}(G)}\sigma_{loc,A}(x,\eta)d\mu_{L}(x,\eta)\neq 0,
\end{equation}then the statement follows from Proposition \ref{Invariance:trace}. On the other hand, assume that $A\in \mathscr{S}_1(L^2(G))$ and that 
\begin{equation}
    \textnormal{Av}[\sigma_{loc,A}]:=\smallint\limits_{ T^{*}\mathbb{S}(G)}\sigma_{loc,A}(x,\eta)d\mu_{L}(x,\eta)= 0.
\end{equation}Define the operator $\tilde{A}=A+(\mathcal{L}_G)^{-n-\varepsilon},$ $\varepsilon>0.$ Note that $\tilde A\in  \mathscr{S}_1(L^2(G)) $ since $\Vert \tilde A\Vert_{\mathscr{S}_1}\leq \Vert  A\Vert_{\mathscr{S}_1}+\Vert \tilde (\mathcal{L}_G)^{-n-\varepsilon}\Vert_{\mathscr{S}_1}<\infty. $ Moreover by definition of $\tilde{A}$ we have
\begin{equation}\label{average}
    \textnormal{Av}[\sigma_{loc,\tilde A}]:=\smallint\limits_{ T^{*}\mathbb{S}(G)}||\eta||^{-n-\varepsilon}_gd\mu_{L}(x,\eta)=\textnormal{Vol}(T^*\mathbb{S}(G))\neq 0.
\end{equation}Note that if $m\geq -n,$ then the order of $\tilde A$ would be larger than $-n$. However,  the condition in  \eqref{average} together with the fact that $\tilde A\in \mathscr{S}_1(L^2(G))$ would imply that $m<-n$ in view of Proposition \ref{Invariance:trace}, which is a contradiction. Thus, one has that $m<-n$ as expected. The proof is complete.   
\end{proof}

Now, we will use the ideal properties of the Schatten classes to derive the extension of Proposition \ref{Invariance:trace} to the case $r>1.$ 
\begin{proposition}\label{Invariance:trace:2} Let $A\in \Psi^{m}_{1,0}(G)$ be a classical pseudo-differential operator of order $m\in \mathbb{R}$. Assume that the average of the principal symbol $\sigma_{loc,A}\in C^\infty(T^*G)$ of $A$ on the co-sphere is non-zero, that is
\begin{equation}\label{Average:symbol:22}
    \textnormal{Av}[\sigma_{loc,A}]:=\smallint\limits_{ T^{*}\mathbb{S}(G)}\sigma_{loc,A}(x,\eta)d\mu_{L}(x,\eta)\neq 0.
\end{equation}
For any $1<r< \infty,$ if $A\in \mathscr{S}_r(L^2(G))$ then $m\leq -n/r.$ Moreover, for $r=\infty,$     $A\in \mathscr{S}_\infty(L^2(G))=\mathscr{B}(L^2(G))$ if and only if $m\leq 0.$
\end{proposition}
\begin{proof}
 Let us consider first, the case where $1<r<\infty.$ Let $\varepsilon>0.$ Consider the operator $\mathcal{L}_G^{s/2}$ where $s$ satisfies
$$ s:=-\frac{n}{q}-\varepsilon<-\frac{n}{q},$$ and where $q:=r/(r-1)$ is the conjugate exponent of $r.$ Then, $q$ is given by the identity 
$1=\frac{1}{q}+\frac{1}{r}.$ Note that since $r>1,$ one has that $q>0.$ Proposition \ref{Invariance:trace} implies that $\mathcal{L}_G^{s/2}\in \mathscr{S}_q(L^2(G)).$ Since $q$ and $r$ and conjugate exponents, we have that
\begin{equation*}
    \tilde{A}:=A\mathcal{L}_G^{s/2}\in \mathscr{S}_1(G).
\end{equation*}Note that the principal symbol $\sigma_{loc,  \tilde{A}}$ satisfies the identity
\begin{equation}\label{symbol:prin}
   \forall(x,\eta)\in T^{*}G\setminus \{0\},\,\, \sigma_{loc,  \tilde{A}}(x,\eta)=\sigma_{loc,  {A}}(x,\eta)\|\eta \|^{s}.
\end{equation}In particular, on the co-sphere we have
\begin{equation}
   \forall(x,\eta)\in T^{*}\mathbb{S},\,\,\,\, \sigma_{loc,  \tilde{A}}(x,\eta)=\sigma_{loc,  {A}}(x,\eta).
\end{equation}Then, \eqref{Average:symbol:22} allows to deduce the non-vanishing property
\begin{equation*}
    \textnormal{Av}[\sigma_{loc,A}]:=\smallint\limits_{ T^{*}\mathbb{S}(G)}\sigma_{loc,\tilde A}(x,\eta)d\mu_{L}(x,\eta)\neq 0,
\end{equation*}allowing the use of the Proposition \ref{Invariance:trace} in order to conclude that the order $m+s$ of $\tilde A\in \mathscr{S}_1$ satisfies the inequality 
$m+s<-n.$ Then, we have that 
$$  m<-s-n=\frac{n}{q}+\varepsilon -n=n\left(\frac{1}{q}-1\right)+\varepsilon =-\frac{n}{r}+\varepsilon\,, $$
where the last hold tru for any $\epsilon>0$.
Taking $\varepsilon\rightarrow 0^+$ we have that $m\leq -n/r.$

Now, let us assume that $r=\infty.$   For $m\leq 0,$ it follows from the Calder\'on-Vaillancourt theorem that $A\in \mathscr{S}_{\infty}(L^2(G))=\mathscr{B}(L^2(G)).$  Now, let us assume that $A\in \mathscr{S}_\infty(L^2(G))=\mathscr{B}(L^2(G))$ and let us give an argument proving that its order satisfies $m\leq 0.$ Proceeding by contradiction suppose that  $m>0.$ Note that the conjugate exponent to $r=\infty$ is $q=1.$ Let $\varepsilon\in (0,m)$ and let $s=-n-\varepsilon.$ Consider the operator $\mathcal{L}_G^{s/2}\in \mathscr{S}_1(L^2(G))$ where $s<-n$. Then  the operator
\begin{align*}
   \tilde{A}:=A\mathcal{L}_G^{s/2}\in \mathscr{S}_\infty(L^2(G))\circ \mathscr{S}_1(L^2(G))\subseteq \mathscr{S}_1(L^2(G)),
\end{align*} is of trace class. Since the principal symbol is given by \eqref{symbol:prin} the co-sphere condition in \eqref{Average:symbol:22} is satisfied and Proposition  \ref{Invariance:trace} implies that the order of $\tilde A$ satisfies the inequality $m+s<-n.$ We have that $$ m+s= m-n-\varepsilon= (m-\varepsilon)-n\geq -n, $$ which contradicts the analysis above. So, we must have that $m\leq 0. $  
 The proof of Proposition \ref{Invariance:trace:2} is complete.
\end{proof}
\begin{corollary}\label{Classificaton:r}
    Assume that $r\in (1,\infty)\cap \mathbb{Z}.$ Let $A\in \Psi^{m}_{1,0}(G)$ be a classical pseudo-differential operator of order $m$. Assume that the average of the principal symbol $\sigma_{loc,A}\in C^\infty(T^*G)$ of $A$ on the co-sphere is non-zero, that is
\begin{equation}\label{Average:symbol:22:cor}
    \textnormal{Av}[\sigma_{loc,A}]:=\smallint\limits_{ T^{*}\mathbb{S}(G)}\sigma_{loc,A}(x,\eta)d\mu_{L}(x,\eta)\neq 0.
\end{equation}
Then, $A\in \mathscr{S}_r(L^2(G))$  if and only if    $m<-n/r.$ 
\end{corollary}
\begin{proof}From Corollary \ref{Schatten:properties:cor} it follows that for $m<-n/r,$ $1<r<\infty,$ we have $A\in \mathscr{S}_r(L^2(G)).$ Now, let us prove the converse statement, that is if $A\in \mathscr{S}_r(L^2(G))$, then $m<-n/r.$ To do this, let us show that the borderline $m=-n/r$ in Proposition \ref{Invariance:trace:2} is not possible. To do this, observe that if $A\in \mathscr{S}_r,$ then
\begin{equation}
    A^r=A\circ A\circ\cdots A\in \mathscr{S}_r(L^2(G))\circ \mathscr{S}_r(L^2(G))\circ\cdots \mathscr{S}_r(L^2(G))\subseteq \mathscr{S}_{1}(L^2(G)),
\end{equation} where the composition is taken $r$-times. However, the order of the trace class operator $A^r$ is $-n$ and this contradicts the conclusion in Proposition \ref{Invariance:trace}. So, necessarily $m<-n/r.$     
\end{proof}

\begin{theorem}\label{Invariance:trace2r} Let $A\in \Psi^{m}_{1,0}(G)$ be a classical  pseudo-differential operator of order $m$.
\begin{itemize}
    \item[(A)] If $r\in [1,\infty)\cap \mathbb{Z},$ then
$A\in \mathscr{S}_r(L^2(G))$ if and only if $m<-n/r.$ For $r=\infty,$     $A\in \mathscr{B}(L^2(G))$ if and only if $m\leq 0.$ 
\item[(B)] If $r\in (1,\infty)\setminus \mathbb{Z},$ and   $A\in \mathscr{S}_{r}(L^2(G)),$ then $m\leq -n/r.$ Moreover, if $A$ is elliptic then one has the strict inequality $m<-n/r.$
\end{itemize}
\end{theorem}
\begin{proof}
Let us prove the statement in (A).
For $r=1$ this statement  has been proved in Theorem \ref{Invariance:trace2}. At this stage of the manuscript we only have to prove that if  $A\in  \mathscr{S}_r(L^2(G)) $ then $m\leq -n/r,$ where $n/r:=0$ for $r=\infty.$  Now, let us consider the case where $1<r\leq \infty,$ with $r\in \mathbb{Z}.$ 

Note that if the average of the principal symbol $\sigma_{loc,A}\in C^\infty(T^*G)$ of $A$ on the co-sphere is non-zero, that is
\begin{equation}
    \textnormal{Av}[\sigma_{loc,A}]:=\smallint\limits_{ T^{*}\mathbb{S}(G)}\sigma_{loc,A}(x,\eta)d\mu_{L}(x,\eta)\neq 0,
\end{equation}then the statement follows from Corollary \ref{Classificaton:r}. On the other hand, assume that $A\in \mathscr{S}_1(L^2(G))$ and that 
\begin{equation}
    \textnormal{Av}[\sigma_{loc,A}]:=\smallint\limits_{ T^{*}\mathbb{S}(G)}\sigma_{loc,A}(x,\eta)d\mu_{L}(x,\eta)= 0.
\end{equation}Define the operator $\tilde{A}=A+(\mathcal{L}_G)^{-n/r-\varepsilon},$ $\varepsilon>0,$ where $n/r:=0$ whenever $r=\infty.$ Note that $\tilde A\in  \mathscr{S}_r(L^2(G)) $ since $\Vert \tilde A\Vert_{\mathscr{S}_r}\leq \Vert \tilde A\Vert_{\mathscr{S}_r}+\Vert \tilde (\mathcal{L}_G)^{-n/r-\varepsilon}\Vert_{\mathscr{S}_r}<\infty. $ On the other hand note that 
\begin{equation}\label{averageP}
    \textnormal{Av}[\sigma_{loc,\tilde A}]:=\smallint\limits_{ T^{*}\mathbb{S}(G)}||\eta||^{-n/r-\varepsilon}_gd\mu_{L}(x,\eta)=\textnormal{Vol}(T^*\mathbb{S}(G))\neq 0.
\end{equation}Now that if $m\geq -n/r,$ then the order of $\tilde A$ must be larger than $-n/r$ (or strictly larger than zero for $r=\infty$) but the condition in  \eqref{averageP} together with the fact that $\tilde A\in \mathscr{S}_r(L^2(G))$  implies that $m<-n/r$ in view of Proposition \ref{Invariance:trace}, which is a contradiction. Thus, one has that $m<-n/r$ for $1<r<\infty$ and $m\geq 0$ for $r=\infty,$ as expected.

Now, let us prove the statement in (B). It can be done following  the proof of Proposition \ref{Invariance:trace:2}. However, let us give another argument. For this, consider $r\in (1,\infty)\setminus \mathbb{Z}.
$ Then, the integer part $[r]$ of $r$ satisfies the strict inequalities
\begin{equation}
    [r]<r<[r]+1.
\end{equation} Let $\varepsilon>0$ and define the parameter
\begin{equation}\label{S:p}
    s=n\left(\frac{1}{[r]}-\frac{1}{r}\right)+\varepsilon.
\end{equation}Since $A\in \mathscr{S}_r
$ and  for any $\varepsilon'>0,$ $(1+\mathcal{L}_G)^{-\frac{s}{2}}\in \mathscr{S}_{\frac{n}{s}+\varepsilon'},$ we will prove that there exists $\varepsilon_0'>0$ such that $A_{s}:=A(1+\mathcal{L}_G)^{-\frac{s}{2}}\in \mathscr{S}_{[r]}.$ To do this, we have to guarantee, that there exists $\varepsilon_0'>0$ such that, the H\"older property 
\begin{equation}\label{Equation:e}
    \frac{1}{[r]}=\frac{1}{r}+\frac{1}{\frac{n}{s}+\varepsilon_0'}
\end{equation}holds because its validity would imply that $A_{s}\in \mathscr{S}_{r}\circ \mathscr{S}_{\frac{n}{s}+\varepsilon_0'}\subseteq \mathscr{S}_{[r]}. $ To prove that the equation in \eqref{Equation:e} admits a solution, define the continuous function 
\begin{equation}
    g(\varepsilon'):=\frac{1}{\frac{n}{s}+\varepsilon'}.
\end{equation} Note that 
\begin{equation}
    g(0)=\frac{s}{n}>\frac{1}{[r]}-\frac{1}{r}>0=g(\infty):=\lim_{\varepsilon'\rightarrow \infty}g(\varepsilon').
\end{equation}By Bolzano's theorem, there exists $\varepsilon_0'\in (0,\infty)$ such that $g(\varepsilon_0')= \frac{1}{[r]}-\frac{1}{r}$ as desired. We have now proved the existence of $\varepsilon_0'$ in in \eqref{Equation:e} and its implication $A_{s}\in \mathscr{S}_{[r]}.$ Hence by the result in Part (A), the order $m-s$ of $A_s,$ satisfies the inequality $m-s<-{n}/{[r]},$ or equivalently $m<s-({n}/{[r]}).$ In view of \eqref{S:p}, we have proved that
\begin{equation}
    \forall \varepsilon>0,\,\,\, m<
    s=n\left(\frac{1}{[r]}-\frac{1}{r}\right)+\varepsilon-\frac{n}{[r]}=-\frac{n}{r}+\varepsilon,
\end{equation}which certainly implies that $m\leq -n/r.$ Note that if $A$ is elliptic, this inequality can be improved to an strict inequality, namely we would have $m<-n/r,$ as proved in Theorem \ref{Schatten:properties:}. The proof is complete.   
\end{proof}

We end this subsection with the following atypical construction in the case of the classes $\Psi^{m}_{0,0}(\mathbb{T}^n
)$ on the $n$-dimensional torus.

\begin{theorem}\label{The:atypical:case}
    For any $\varkappa>0,$ there exists a {non-elliptic} pseudo-differential operator $A$ in the exotic class   $\Psi^{-\varkappa}_{0,0}(\mathbb{T}^n
)\setminus \Psi^{-\varkappa-\varepsilon}_{0,0}(\mathbb{T}^n
), $ for all $\varepsilon>0,$ that belongs to all the Schatten ideals $\mathscr{S}_r(L^2(\mathbb{T}^n)),$ where $r>0.$
\end{theorem}
\begin{proof}Let $\varkappa>0,$ and let $(e_{j})_{j=1}^n$ be the canonical orthonormal basis of $\mathbb{R}^n,$  and consider the set 
\begin{equation}
    \mathbb{D}=\{ 2^{k}e_{1}=(2^k,0,\cdots, 0):k\in \mathbb{N}
    \}.
\end{equation}Define the symbol
\begin{equation}
    a(\xi)=\langle\xi\rangle^{-\varkappa},\,\,\xi\in \mathbb{D};\,\,a(\xi)=0,\,\xi\in \mathbb{Z}^n\setminus \mathbb{D}.
\end{equation}Observe that for any $\alpha\in \mathbb{N}_0^n,$ we have
\begin{align*}
    |\Delta^\alpha a(\xi)|\leq \sum_{\beta\leq \alpha}{{\alpha}\choose{\beta}}|a(\xi+\beta)|\leq \sum_{\beta\leq \alpha}{{\alpha}\choose{\beta}}\langle \xi+\beta\rangle^{-\varkappa}\lesssim_{\alpha}(1+|\xi|)^{-\varkappa },\,\,\xi\in \mathbb{Z}^n. 
\end{align*}To see that the right-hand side of this inequality  is the best possible, we consider the case where $\alpha=e_1$ and when $\xi_k=2^{k}e_{1}\in \mathbb{D}.$ Indeed, note that
\begin{align*}
    |\Delta^{e_1}a(\xi_k)| &=|a(\xi_k+e_1)-a(\xi_k)|=|\langle 2^ke_1 \rangle^{-\varkappa}|
    = \langle \xi_k\rangle^{-\varkappa}  .
\end{align*} This argument proves that the estimate $|\Delta^{e_1}a(\xi)| \leq \langle \xi\rangle^{-\varkappa}  ,$ is sharp and even that it is an equality when $\xi_k\in \mathbb{D}.$    

To prove that the operator $A=\textnormal{Op}(a)$ associated to the symbol $a$ belongs to the Schatten class $\mathscr{S}_{r}(L^2(\mathbb{T}^n))$ note that the sequence of singular values of $A$ is determined by the values of its symbol, that is
\begin{equation}
    s_{\xi}(A)=|a(\xi)|,\,\,\,\xi\in \mathbb{Z}^n,
\end{equation} are all the singular values of $A.$  Consequently
\begin{align*}
    \sum_{\xi\in \mathbb{Z}^n
    }s_\xi(A)^r &=\sum_{\xi_k\in \mathbb{D}} s_\xi(A)^r
\asymp \sum_{k=1}^\infty 2^{-k\varkappa r}<\infty.
\end{align*}Now, to prove that $A$ is not elliptic and that it is of order $m=-\varkappa,$ observe that the best estimate from below that $a$ satisfies is the following
\begin{equation}
    |a(\xi)|\geq 0,
\end{equation}with equality when $\xi\in\mathbb{Z}^n\setminus \mathbb{D}. $ The proof of this atypical case is complete.    
\end{proof}

\subsection{Classical operators in Schatten classes and their principal symbols}\label{Classical:symbols:section}

We consider the subclass $\Psi_0^m(G)$ in  $\Psi_{cl}^{m/2}(G)$  of
operators with homogeneous symbols of
order $m.$  First we have
the following result where we show that only the  principal  homogeneous component of the principal symbols contains the information about the membership of a classical operator to any Schatten class $\mathscr{S}(L^2(G)).$

\begin{proposition}\label{Prop:RedToHom}
Let $r\in (0,\infty ]$ and $m\in \mathbb{R}$. Then
the following conditions are equivalent:
\begin{enumerate}
\item $\Psi ^m_{cl}(G)\subseteq \mathscr{S}_r(L^2(G))$;

\item $\Psi ^m_0(G)\subseteq \mathscr{S}_r(L^2(G))$.
\end{enumerate}
\end{proposition}

\begin{proof}
Evidently, (1) implies (2). Suppose that (2) holds.
We shall prove that (1) is true. If $N$ is large enough,
then $\Psi ^{m-N}\subseteq \mathscr{S}_r(L^2(G))$. Observing that the operator $B=\sqrt{\mathcal{L}}_G^{-1}$ is continuous on $L^2(G)$, 
we get that $\Psi _0^mB^k\in \mathscr{S}_r(L^2(G))$ for any $k\in \mathbb{N}$. Since, the principal symbol of $B^{k},$ which is given by $$\sigma_{B^k}(x,\eta)=\Vert \eta\Vert_{g}^{-k},\,\,\forall x\in G,\,\forall \eta\in \mathfrak{g}\setminus \{0\},$$   is homogeneous of order $-k,$ we have that
$$
\Psi _0^mB^k = \Psi _0^{m-k}\, 
\operatorname{mod}\, \Psi^{-\infty}
(G\times \widehat G).
$$
Note that $\textnormal{Op}(\Psi^{-\infty} (G\times \widehat G))\subseteq \mathscr{S}_r(L^2(G)).$ Then, we get
$$
\Psi _{cl}^m = \sum _{k=0}^{N-1}\Psi _0^mB^k
+ \Psi _{cl}^{m-N}\, 
\operatorname{mod}\, \Psi^{-\infty}
(G\times \widehat G)
\subseteq \mathscr{S}_r(L^2(G)),
$$ in view of our hypothesis  $\Psi ^m_0\subseteq \mathscr{S}_r(L^2(G)),$
which shows that (1) holds when (2) holds. This
gives the result.
\end{proof}

In view of Theorem \ref{Invariance:trace2r} we have the following characterisation.
\begin{corollary}\label{Invariance:trace2r:cor} Let $A\in \Psi^{m}_{0}(G)$ be of order $m\in \mathbb{R}$.
\begin{itemize}
    \item[(A).] If $r\in [1,\infty)\cap \mathbb{Z},$ then
$A\in \mathscr{S}_r(L^2(G))$ if and only if $m<-n/r.$ For $r=\infty,$     $A\in \mathscr{B}(L^2(G))$ if and only if $m\leq 0.$ 
\item[(B).] If $r\in (1,\infty)\setminus \mathbb{Z},$ and   $A\in \mathscr{S}_{r}(L^2(G)),$ then $m\leq -n/r.$ Moreover, if $A$ is elliptic then one has the strict inequality $m<-n/r.$
\end{itemize}
\end{corollary}

\subsection{Order of classical operators vs order of global symbols}\label{Last:Subsection}
As an application of the methods developed in this manuscript we prove that the order of a symbol characterises the order of the operator.

\begin{proposition}\label{Toft:criterion}
    Let $\mu,t\in \mathbb{R}$ and let $A\in \Psi^{\mu}_{1,0}(G)$ be a classical pseudo-differential operator such that its global symbols satisfies
    \begin{equation}\label{growth:t2}
      \forall(x,[\xi])\in G\times \widehat{G},\,  \Vert \sigma_A(x,[\xi])\Vert_{\textnormal{op}}\leq C\langle\xi\rangle^{t}.
    \end{equation}Then, $A\in \Psi^{t}_{1,0}(G),$ provided that   the average of its principal symbol $\sigma_{loc,A}\in C^\infty(T^*G)$ on the co-sphere is non-zero:
\begin{equation}\label{Average:symbol:22:2}
    \textnormal{Av}[\sigma_{loc,A}]:=\smallint\limits_{ T^{*}\mathbb{S}(G)}\sigma_{loc,A}(x,\eta)d\mu_{L}(x,\eta)\neq 0.
\end{equation}
\end{proposition}
\begin{proof}
   Let us consider two cases: $t=0$ and $t\neq 0.$ 
    \begin{itemize}
        \item {\bf Case 1: $t=0.$} Assume that $ A\in \Psi^{\mu}_{1,0}(G).$ It is clear that if $\mu\leq 0$ we have nothing to prove. So, the relevant case is when $\mu>0.$ First, let us prove the statement in the case where $\mu \in \mathbb{N}.$ To do it, we will proceed using induction. So, let us start by proving that the result is true if $\mu=1.$ Then, assuming that $A\in \Psi^{1}_{1,0}(G)$ and with the additional condition  
   \begin{equation}\label{bounded:symbol}
      \forall(x,[\xi])\in G\times \widehat{G},\,  \Vert \sigma_A(x,[\xi])\Vert_{\textnormal{op}}\leq C,
    \end{equation} we have to deduce that $A\in \Psi^0_{1,0}(G).$ To prove this, we will use Corollary 2.2 of \cite{RuzhanskyTurunen2011} to deduce that $A$ is bounded on $L^2(G),$ and then from   Proposition \ref{Invariance:trace:2} we use the $L^2(G)$-boundedness of $A$ and the average condition in \eqref{Average:symbol:22:2} to deduce that its order $\mu\leq 0.$ This proves the case $\mu=1.$ Now, let us state our induction hypothesis:
    \begin{itemize}
        \item Let us assume that $\mu_0\in \mathbb{N}$ is such that
    any classical pseudo-differential operator $\tilde A\in \Psi^{\mu_0}_{1,0}(G)$ with a bounded operator norm of its matrix-valued symbol
    \begin{equation}\label{bounded:symbol:2}
      \forall(x,[\xi])\in G\times \widehat{G},\,  \Vert \sigma_{\tilde A}(x,[\xi])\Vert_{\textnormal{op}}\leq C,
    \end{equation} and with its principal symbol satisfying the average condition \eqref{Average:symbol:22:2} belongs to the class $\Psi^{0}_{1,0}(G).$
    \end{itemize}
    Now, let $A\in \Psi^{\mu_0+1}_{1,0}(G)$ be a classical pseudo-differential operator whose symbol satisfies \eqref{Average:symbol:22:2} and \eqref{bounded:symbol}.  Define $\tilde{A}=A(1+\mathcal{L}_G)^{-\frac{1}{2}}.$ Note that $\tilde A\in \Psi^{r_0}_{1,0}(G)$ and that its symbol satisfies \eqref{bounded:symbol:2} since
    \begin{align*}
        \Vert \sigma_{ \tilde{A}}(x,[\xi])\Vert_{\textnormal{op}}=\Vert \sigma_{{A}}(x,[\xi])\langle\xi\rangle^{-1}\Vert_{\textnormal{op}}\leq  \Vert \sigma_{A}(x,[\xi])\Vert_{\textnormal{op}}\leq C.
    \end{align*}Our induction hypothesis implies that $\tilde{A}$ belongs to the class $\Psi^0_{1,0}(G).$ But then, the Calder\'on-Vaillancourt theorem implies that $\tilde A$ is bounded on $L^2(G).$ Since on the co-sphere $T^*\mathbb{S}(G)$ the principal symbols of $\tilde A$ and of $A$ agree, $\tilde A$ satisfies the average condition \eqref{Average:symbol:22:2}.  Then, Proposition \ref{Invariance:trace:2} implies that $\mu_0\leq 0.$ Since
    $$ A(1+\mathcal{L}_G)^{-\frac{1}{2}}\in \Psi^{\mu_0}_{1,0}(G)\subseteq \Psi^{0}_{1,0}(G), $$ then $$ A\in \Psi^{\mu_0+1}_{1,0}(G)\subseteq \Psi^{1}_{1,0}(G), $$
    and together with the hypothesis \eqref{bounded:symbol} we can use the base case $\mu=1$ in the inductive process to establish that $A\in \Psi^0_{1,0}(G).$ So, by the mathematical induction we have proved the statement in Theorem \ref{Toft:criterion} in the case where $\mu\in \mathbb{N}.$ Now, in the general case  if $[\mu]$ is the integer part of every $\mu>0,$ and if $A\in \Psi^{\mu}_{1,0}(G)$ is a classical pseudo-differential operator satisfying \eqref{bounded:symbol}, we have that $A\in \Psi^{[\mu]+1}_{1,0}(G)$ and then the better conclusion $A\in \Psi^{0}(G)$ follows from the statement of Theorem \ref{Toft:criterion}  for integer orders. Thus, we have proved Theorem \ref{Toft:criterion} in the case $t=0.$
\item {\bf Case 2. $t\neq 0.$} Assume that $A\in \Psi^{\mu}_{1,0}(G)$ has a matrix-valued symbol satisfying \eqref{growth:t} and the average condition in \eqref{Average:symbol:22:2}. Then, $\tilde{A}=A(1+\mathcal{L}_G)^{-\frac{t}{2}}\in \Psi^{r-t}_{1,0}(G)$ satisfies \eqref{Average:symbol:22:2} and its matrix-valued symbol satisfies the inequality \eqref{bounded:symbol}. From the first part of the poof, we deduce that $\tilde{A}=A(1+\mathcal{L}_G)^{-\frac{t}{2}}\in \Psi^{0}_{1,0}(G).$ Then, the pseudo-differential calculus implies that $A\in \Psi^{t}_{1,0}(G).$
 \end{itemize}
 The proof of Proposition \ref{Toft:criterion} is complete.  
\end{proof}
 Now, we will remove the geometric average condition in  Proposition \ref{Toft:criterion} to prove a general statement on compact Lie groups.

\begin{theorem}\label{Toft:criterion:II}
    Let $\mu,t\in \mathbb{R}$ and let $A\in \Psi^{\mu}_{1,0}(G)$ be a classical pseudo-differential operator such that its global symbol satisfies
    \begin{equation}\label{growth:t}
      \forall(x,[\xi])\in G\times \widehat{G},\,  \Vert \sigma_A(x,[\xi])\Vert_{\textnormal{op}}\leq C\langle\xi\rangle^{t}.
    \end{equation}Then $A\in \Psi^{t}_{1,0}(G).$ 
\end{theorem}\begin{proof}
    Note that if  the average of the principal symbol $\sigma_{loc,A}\in C^\infty(T^*G)$ of $A$ on the co-sphere is non-zero:
\begin{equation}
    \textnormal{Av}[\sigma_{loc,A}]:=\smallint\limits_{ T^{*}\mathbb{S}(G)}\sigma_{loc,A}(x,\eta)d\mu_{L}(x,\eta)\neq 0,
\end{equation} the statement follows from  Proposition \ref{Toft:criterion}. On the other hand, if the principal symbol of $A$ has average zero on $T^*\mathbb{S}(G),$ that is, \begin{equation}
    \textnormal{Av}[\sigma_{loc,A}]:=\smallint\limits_{ T^{*}\mathbb{S}(G)}\sigma_{loc,A}(x,\eta)d\mu_{L}(x,\eta)= 0,
\end{equation}we define the operator $\tilde A:=A+(\mathcal{L}_G)^{\frac{t}{2}}.$ Is it clear that the principal symbol of $\tilde{A}$ is given by
$$  \sigma_{loc, \tilde A}:= \sigma_{loc,  A}+\Vert \xi\Vert_{g}^{t}. $$ Note that the matrix-valued symbol of $\tilde A$ satisfies also 
\begin{equation}
  \sigma_{\tilde A}(x,[\xi])=  \sigma_{ A}(x,[\xi])+|\xi|^tI_{d_\xi},\,(x,[\xi])\in G\times \widehat{G}.
\end{equation}From our hypothesis we have  the estimate
\begin{equation*}
    \Vert\sigma_{\tilde A}(x,[\xi])\Vert_{\textnormal{op}}\leq \tilde C\langle\xi \rangle^t, \,(x,[\xi])\in G\times \widehat{G},
\end{equation*}and the average condition 
\begin{equation}
    \textnormal{Av}[\sigma_{loc,\tilde A}]:=\smallint\limits_{ T^{*}\mathbb{S}(G)}\sigma_{loc,A}(x,\eta)d\mu_{L}(x,\eta)+\smallint\limits_{ T^{*}\mathbb{S}(G)}||\eta||_g^{t}d\mu_{L}(x,\eta)=\textnormal{Vol}(T^*G)\neq  0.
\end{equation}From Proposition  \ref{Toft:criterion} follows that $\tilde A \in \Psi^{t}_{1,0}(G)$ and using the property $ (\mathcal{L}_G)^{\frac{t}{2}} \in \Psi^{t}_{1,0}(G),$ one obtains that $A=\tilde A-(\mathcal{L}_G)^{\frac{t}{2}} \in \Psi^{t}_{1,0}(G).$ The proof is complete.
\end{proof}

\bibliographystyle{amsplain}

\end{document}